\documentclass[reqno,12pt]{amsart}
\usepackage{latexsym}
\usepackage{amsfonts}
\usepackage{amssymb}
\usepackage{mathrsfs}
\usepackage[all]{xy}
\usepackage{bbm}
\usepackage{enumitem}
\usepackage{color}
\usepackage{mathtools}
\usepackage[mathscr]{euscript}
\usepackage{tikz}
\usepackage{amsmath,amsthm,amssymb,mathrsfs,amsfonts,verbatim,color,leftidx}
\usepackage{amsbsy}
\usepackage[colorlinks,urlcolor=blue,linkcolor=blue,citecolor=brown]{hyperref}
\usetikzlibrary{matrix,arrows}
\usepackage{color}
\definecolor{jasper}{rgb}{0,0.5,0}

\newtheorem{theorem}{Theorem}[section]
\newtheorem{corollary}[theorem]{Corollary}
\newtheorem{lemma}[theorem]{Lemma}
\newtheorem{conjecture}[theorem]{Conjecture}
\newtheorem{problem}[theorem]{Problem}

\newtheorem{proposition}[theorem]{Proposition}
\theoremstyle{definition}
\newtheorem{example}{Example}

\newtheorem{remark}[theorem]{Remark}
\numberwithin{equation}{section}

\newcommand{\RR}{\mathbb{R}}
\newcommand{\TT}{\mathbb{T}}
\newcommand{\CC}{\mathbb{C}}
\newcommand{\NN}{\mathbb{N}}

\newcommand{\QQ}{\mathbb{Q}}
\newcommand{\ZZ}{\mathbb{Z}}
\newcommand{\PP}{\mathbb{P}}

\newcommand{\cR}{\mathcal{R}}

\newcommand{\cP}{\mathcal{P}}

\newcommand{\cD}{\mathcal{D}}

\newcommand{\cL}{\mathcal{L}}

\newcommand{\bQ}{\mathbf{Q}}

\newcommand{\diag}{\mathrm{diag}}
\newcommand{\dimH}{\mathrm{\dim_H}}

\newcommand{\id}{\mathrm{id}}
\newcommand{\dist}{\mathrm{dist}}

\newcommand{\SL}{\mathrm{SL}}

\newcommand{\Span}{\mathrm{Span}}

\newcommand{\la}{\langle}
\newcommand{\ra}{\rangle}

\newcommand{\Bad}{\mathbf{Bad}}
\newcommand{\Mad}{\mathbf{Mad}}

\newcommand {\ignore}[1]  {}

\newif\ifdraft\drafttrue

\draftfalse

\renewcommand{\emptyset}{\varnothing}
\renewcommand{\setminus}{\smallsetminus}

\begin{document}

\title[The set of counterexamples to ULC is HAW]{Winning property of counterexamples to\\ the uniform Littlewood's Conjecture}

\author{Vasiliy Neckrasov}
\address{Department of Mathematics,
  Brandeis University,
  Waltham, MA 02453, USA
}
\email{vneckrasov@brandeis.edu}

\author{Chengyang Wu}
\address{Department of Mathematics, University of Chicago, Chicago, Illinois, 60637, U.S.}
\email{chengyangwu1999@gmail.com}

\author{Bohan Yang}
\address{Shanghai Institute for Mathematics and Interdisciplinary Sciences, Shanghai \indent 200433, China}
\address{Research Institute of Intelligent Complex Systems, Fudan University, Shanghai \indent 200433, China} 
\email{bhyang@simis.cn}

\date{September 2, 2026}

\begin{abstract}
    In this paper, we prove that the set of counterexamples to the uniform Littlewood's conjecture proposed in \cite{BFK25}, that is, the set of pairs of real numbers $(x,y)$ satisfying $$
    \limsup_{Q\to +\infty}Q\cdot \min_{1\leq q\leq Q}\la qx\ra\la qy\ra>0,
    $$
    is hyperplane absolute winning. We show that a stronger statement holds: the set of real pairs $(x,y)$ satisfying $$
    \liminf_{m\to +\infty}Q_m\cdot \min_{1\leq q\leq Q_m}\la qx\ra\la qy\ra>0,
    $$
    is hyperplane absolute winning if $Q_{m+1} \gg Q_m^{\tau}$ for some $\tau > 1$.
    In particular, the above sets have full Hausdorff dimension in $\RR^2$.
    In addition, we prove that these sets are absolute winning on every regular $C^2$ planar curve whose set of points of nonzero curvature is itself absolute winning on the curve. We also establish analogous results for certain lines.
\end{abstract}

\maketitle

\section{Introduction}

\subsection{Littlewood's conjecture and its uniform version}\label{SS:LCULC}
For $x\in\RR$, let $\la x\ra$ denote its distance to the nearest integer. In the 1930s, J. E. Littlewood raised the following conjecture: 

\begin{conjecture}[Littlewood]
    For every $(x,y)\in\RR^2$, 
\begin{equation}\label{E:LC}
	\liminf_{q\to +\infty}q\cdot \la qx\ra\la qy\ra=0.
\end{equation}
\end{conjecture}

Littlewood's conjecture remains far from resolved. It follows from a theorem of Gallagher \cite{Gal62} that almost all pairs $(x,y)$ satisfy \eqref{E:LC}. The best progress towards Littlewood's conjecture was obtained by Einsiedler, Katok, and Lindenstrauss in \cite{EKL06}, where they proved that the set of counterexamples to \eqref{E:LC} has zero Hausdorff dimension.

Recently, Bandi, Fregoli, and Kleinbock in \cite{BFK25} proposed a uniform version of Littlewood's conjecture. In the two-dimensional case, they asked whether every $(x,y)\in\RR^2$ satisfies an even stronger condition than \eqref{E:LC}:
\begin{equation}\label{E:ULC}
	\limsup_{Q\to +\infty}Q\cdot \min_{1\leq q\leq Q}\la qx\ra\la qy\ra=0.
\end{equation}
In the same paper, the authors proved that the set of counterexamples to \eqref{E:ULC} has zero Lebesgue measure. However, they were inclined to believe that there should exist some $(x,y)\in\RR^2$ satisfying \eqref{E:LC} but not \eqref{E:ULC}.

It was subsequently proved by Schleischitz \cite{Sch26} that the above two-dimensional uniform Littlewood's conjecture is false; moreover, its counterexamples form a residual set in $\RR^2$. A simpler proof was given by Moshchevitin in \cite{Mos26}. 

A natural follow-up question is to investigate the Hausdorff dimension of the set of counterexamples to \eqref{E:ULC}. 
When the first version of the present paper was submitted, Shulga \cite{Shu26} independently announced a lower bound $3/2$ for this dimension.


In this paper, our main theorem shows that the set of counterexamples to \eqref{E:ULC} in fact has full Hausdorff dimension in $\RR^2$. Moreover, this will be derived from the much stronger fact that this set is \textit{hyperplane absolute winning} (HAW) on $\RR^2$; see Section \ref{S:Schmidt} for the relevant definitions and properties.

\begin{theorem}\label{T:main} The set of counterexamples to \eqref{E:ULC}
	\begin{equation}\label{E:noULC}
		E:=\left\{
        (x,y)\in\RR^2: \limsup_{Q\to +\infty}Q\cdot \min_{1\leq q\leq Q}\la qx\ra\la qy\ra>0
        \right\}
	\end{equation}
    is HAW on $\RR^2$. In particular, it has full Hausdorff dimension in $\RR^2$.
\end{theorem}


For comparison, we denote the set of counterexamples to \eqref{E:LC} by
    \begin{equation}\label{E:noLC}
		E':=\left\{
        (x,y)\in\RR^2: \liminf_{q\to +\infty}q\cdot\la qx\ra\la qy\ra>0
        \right\}.
	\end{equation}
It follows from preceding discussion that $\dimH{E'}=0$. 
Note that (see Lemma \ref{L:HAW-mnfd}, (iv) and (ii)):
\begin{itemize}
    \item  the complement of any set with zero Hausdorff dimension is HAW; 
    \item HAW sets are stable under countable intersections.
\end{itemize}
So in view of Theorem \ref{T:main}, the next corollary is immediate:

\begin{corollary}\label{C:LCnotULC}
    The difference set $$
    E\setminus E'=\left\{
    (x,y)\in\RR^2: (x,y)\text{ satisfies }\eqref{E:LC} \text{ but not }\eqref{E:ULC}
    \right\}
    $$ 
    is HAW on $\RR^2$. In particular, it has full Hausdorff dimension in $\RR^2$.
\end{corollary}

Corollary \ref{C:LCnotULC} gives an affirmative answer to \cite[Question 1.5 (ii)]{KW25} in terms of Hausdorff dimension. Moreover, we may intersect $E\setminus E'$ with more HAW sets, and the resulting set remains HAW and hence has full Hausdorff dimension. For example, given a pair $(s,t)$ of nonnegative numbers with $s+t=1$, we recall that (see \cite{An16,NS14}) the set of $(s,t)$-badly approximable vectors \begin{equation}\label{E:bad-definition}
    \Bad(s,t):=\left\{
        (x,y)\in\RR^2: \inf_{q\in\NN}\max\{q^{s}\la qx\ra,q^t\la qy\ra\}>0
        \right\}
\end{equation}
is HAW on $\RR^2$. Then we have the following corollary:

\begin{corollary}\label{C:BadLCnotULC}
    Let $(s_n,t_n)_{n=1}^{+\infty}$ be a sequence of pairs of nonnegative numbers with $s_n+t_n=1$. Then the set $$
    (E\setminus E')\cap \bigcap_{n=1}^{+\infty}\Bad(s_n,t_n)
    $$
    is HAW on $\RR^2$. In particular, it has full Hausdorff dimension in $\RR^2$.
\end{corollary}

\subsection{A refined version of Theorem \ref{T:main} for fixed sequences}\label{SS:refined}
In this subsection, we develop a general framework that incorporates both Littlewood's conjecture and its uniform version. From this perspective, one may see why our results provide quantitative insights into Littlewood's conjecture.

Let $\bQ=(Q_m)_{m\geq 0}$ be a strictly increasing sequence of positive real numbers with $Q_0>1$ and $Q_m\to\infty$. We consider the following set: 
\begin{equation}\label{E:EQ}
    E_{\bQ}:=\left\{
(x,y)\in\RR^2:\liminf_{m\to +\infty}Q_m\cdot \min_{1\leq q\leq Q_m}\la qx\ra\la qy\ra>0
\right\}.
\end{equation}
As shown in Lemma \ref{L:LCULC}, for any such sequence $\bQ$ one has $$
E'\subseteq E_{\bQ}\subseteq E.
$$
We have seen from Section \ref{SS:LCULC} that \begin{equation}\label{E:contrast}
    \dimH{E'}=0,\qquad \text{while}\qquad \dimH{E}=2\quad\text{(in fact, $E$ is HAW on $\RR^2$)}.
\end{equation}
So it is tempting to know if one could reinterpret \eqref{E:contrast} in terms of properties of the intermediate sets $E_{\bQ}$. Motivated by this, we propose the following question: 

\begin{problem}\label{P:growth}
Given a growth rate of the sequence $\bQ$, what is $\dimH{E_{\bQ}}$?
\end{problem}

For example, when this sequence $\bQ$ has bounded ratios, one must have $E'=E_{\bQ}$ (see Lemma \ref{L:LCULC}), whose Hausdorff dimension is zero. On the other hand, when this sequence $\bQ$ grows \textit{doubly exponentially}, we are able to prove the following result, along the lines of the proof of Theorem \ref{T:main}: 

\begin{theorem}\label{T:1+epsilon}
    Let $\bQ=(Q_m)_{m\geq 0}$ be any strictly increasing sequence of positive real numbers with $Q_0>1$ and $Q_m\to\infty$, satisfying \begin{equation}\label{E:1+epsilon}
        \inf_{m\geq 0}\frac{Q_{m+1}}{Q_m^{\tau}}>0\quad \text{for some }\tau>1.
    \end{equation}
    Then $E_{\bQ}$ is HAW on $\RR^2$. In particular, $\dimH{E_{\bQ}}=2$.
\end{theorem}

So we summarize here the above answers to Problem \ref{P:growth}: \begin{equation*}
    Q_m\nearrow +\infty\text{ and }Q_{m+1}\asymp Q_m^{\tau}\Longrightarrow \begin{cases}
        \dimH{E_{\bQ}}=0, &\text{ if }\tau=1\\
        E_{\bQ}\text{ is HAW}, &\text{ if }\tau>1
    \end{cases}.
\end{equation*}
An important class of remaining unsolved cases consists of sequences for which \begin{equation}\label{E:lograte}
    \frac{Q_{m+1}}{Q_m}\to +\infty\quad\text{but}\quad \frac{\log{Q_{m+1}}}{\log{Q_m}}\to 1.
\end{equation}
This exactly corresponds to \cite[Problem 1]{Sch26}. For example, a natural growth rate of $\bQ=(Q_m)_{m\geq 0}$ satisfying \eqref{E:lograte} is \begin{equation}\label{E:lograte0}
Q_m\nearrow +\infty\quad\text{and}\quad Q_{m+1}\asymp Q_m(\log{Q_m})^{\lambda} \text{ for some }\lambda>0.
\end{equation}
The authors believe that this is worth investigating. 

Finally, we compare our general framework with those of \cite{Bad13,BV11,FK26}. Given a function $f:\NN\to \RR_+$, define the set\footnote{The name $\Mad$ was proposed by Badziahin and Velani in \cite{BV11} and stands for multiplicatively badly approximable.}  \begin{equation}
    \Mad(f):=\{(x,y)\in\RR^2:\liminf_{q\to +\infty}qf(q)\cdot \la qx\ra\la qy\ra>0
    \}.
\end{equation}
For a strictly increasing sequence $\bQ=(Q_m)_{m\geq 0}$ tending to infinity, if we set \begin{equation*}
    f_{\bQ}(q):=\frac{\min\{Q_m:m\geq 1,\;q\leq Q_m\}}{q},
\end{equation*}
then it is straightforward to see that \begin{equation}
    E_{\bQ}=\Mad(f_{\bQ}).
\end{equation}
In particular, when the sequence $\bQ$ satisfies \eqref{E:lograte0}, one has $E_{\bQ}\subseteq \Mad((\log{})^{\lambda}).$



\subsection{Winning property on $C^2$ curves and lines}\label{SS:curve-results}

The HAW property of the sets $E$ and $E_{\bQ}$ established above shows their largeness in several ways; for instance, HAW sets are known to have full Hausdorff dimension on certain fractals (see, for example, \cite{BHNS25,BFKRW12}). In this section we prove that these sets are \textsl{absolutely winning} (this property is in general stronger than HAW, but they are equivalent in the one-dimensional case) on $C^2$ curves. One may see \cite{Mc10} for the original McMullen's absolute game and its winning sets.

Historically, there has been considerable interest in studying winning properties of Diophantine sets restricted to curves. The origins of this direction may be traced back to Davenport \cite{Dav64}; we refer to \cite{BV14} for a detailed historical account and further discussion.

The absolute winning property of badly approximable vectors on curves, including $\Bad(s,t)$ and its higher-dimensional generalization, has been established in \cite{ABV18,BNY22}. To the best of the authors' knowledge, no analogous result was obtained in the multiplicative setting. The main theorem of this section provides the first result of this kind.

Throughout this paper, by a regular $C^2$ curve $\mathcal C\subset\RR^2$ we mean a one-dimensional embedded $C^2$ submanifold. We will denote the curvature of $\mathcal{C}$ at the point $z \in \mathcal{C}$ by $\kappa(z)$.

\begin{theorem}\label{T:main-curves}
Let $\bQ=(Q_m)_{m\geq0}$ be as in Theorem \ref{T:1+epsilon}, and let $\mathcal C\subset\RR^2$ be a regular $C^2$ curve
satisfying that the set $\{ z \in \mathcal{C}: \kappa(z) \neq 0 \}$ is absolutely winning on $\mathcal{C}$.
Then the sets $E\cap\mathcal C$ and $E_{\bQ}\cap \mathcal{C}$ are absolutely winning on $\mathcal{C}$. 
In particular, their intersection with every nonempty open arc of $\mathcal C$ has Hausdorff dimension one.
\end{theorem}

The assumption of the above theorem holds if the set of points with zero curvature has zero Hausdorff dimension, e.g. is countable (see Lemma \ref{L:HAW-mnfd} (iv)).



We note that Theorem \ref{T:main-curves} 
is independent of Theorems \ref{T:main} and \ref{T:1+epsilon}. In fact, HAW property on $\RR^2$ does not imply HAW property on curves, and vice versa, HAW property on each curve does not imply HAW on $\RR^2$. For counterexamples in both directions, see Appendix \ref{appendix1}.

\smallskip
At the end of this section, we study the sets $E$ and $E_{\bQ}$ on certain lines.

\begin{theorem}\label{T:lines}
    Let $\bQ=(Q_m)_{m\geq0}$ be as in Theorem \ref{T:1+epsilon}, and let
$\mathcal C = \{ (x,y) \in \RR^2: y = ax + b \}$.
\begin{enumerate}
    \item If $a \neq 0$ and for some $\varepsilon > 0$ one has
    \begin{equation}\label{E:ABV_condition}
    \liminf_{q\to +\infty}q^{2 - \varepsilon}\cdot \max \{ \la qa\ra, \la qb\ra \}>0,
    \end{equation}
    then the sets $E \cap \mathcal C$ and $E_{\bQ} \cap \mathcal C$ are absolutely winning on $\mathcal C$.
    \item If $a, b \in \QQ$, then $\mathcal C \cap E = \mathcal C \cap E_{\bQ} = \varnothing.$
\end{enumerate}
\end{theorem}

The authors are not aware of results for the pairs $(a, b)$ not covered by Theorem \ref{T:lines}.

\subsection{Relationship to homogeneous dynamics} As is well known, Littlewood's conjecture and its uniform version can be reformulated in the language of homogeneous dynamics. To $(x,y)\in\RR^2$ one can attach a point $p_{x,y}$ in $X_3=\SL_3(\RR)/\SL_3(\ZZ)$: \begin{equation*}
    p_{x,y}=\begin{pmatrix}
        1 &0 &x\\ 0 &1 &y \\ 0 &0 &1
    \end{pmatrix}\cdot \SL_3(\ZZ),
\end{equation*}
and consider its orbit under left translations by elements in $$
A^+:=\{\diag(e^s,e^t,e^{-(s+t)}):s,t\geq 0\}.
$$
Using Mahler’s compactness criterion, it was shown in \cite{EKL06} that $(x,y)$ satisfies \eqref{E:LC} if and only if the orbit $A^+p_{x,y}$ is unbounded.

Furthermore, we consider a partition of $A^{+}$ into fibers 
\begin{equation*}
    L_T^{+}:=\left\{\diag(e^s,e^t,e^{-(s+t)}):s,t\geq 0,s+t=T\right\},
\end{equation*}
where $T$ runs through all non-negative real numbers. Then we say that the orbit $A^{+}p$, where $p\in X_3$, is \textsl{fiberwise divergent} if for every compact subset $K$ of $X_3$ there exists $T_0 > 0$ such that $L_T^+p\not\subseteq K$ for all $T > T_0$.  It was proved in  \cite{BFK25} that $(x,y)$ satisfies \eqref{E:ULC} if and only if the orbit $A^{+}p_{x,y}$ is fiberwise divergent. 

Besides, given a pair $(s,t)$ of nonnegative numbers with $s+t=1$, we also consider a ray inside $A^+$ of the form $$
F_{(s,t)}^+:=\{\diag(e^{su},e^{tu},e^{-u}):u\geq 0\}.
$$
It is well known that (see \cite{Da1, K98}) $(x,y)\in \Bad(s,t)$ if and only if the orbit $F_{(s,t)}^+p_{x,y}$ is bounded.

Therefore, using the local diffeomorphism $(x,y)\mapsto p_{x,y}$, we may translate Theorem \ref{T:main}, Corollary \ref{C:BadLCnotULC}, and Theorem \ref{T:main-curves} into corresponding dynamical statements on the two-dimensional subtorus of $X_3$: \begin{equation*}
    \TT:=\{p_{x,y}:(x,y)\in\RR^2\}.
\end{equation*}

\begin{theorem}
    The set $$
    \{
    p\in\TT: A^+p\text{ is not fiberwise divergent}
    \}
    $$
    is HAW on $\TT$. In particular, it has full Hausdorff dimension in $\TT$.
\end{theorem}

\begin{corollary}
Let $(s_n,t_n)_{n=1}^{+\infty}$ be a sequence of pairs of nonnegative numbers with $s_n+t_n=1$. Then the set $$
    \left\{
    p\in\TT: 
    \begin{aligned}
    &A^+p\text{ is not fiberwise divergent and unbounded},\\
    &\text{and for all }n\geq 1,\;F^+_{(s_n,t_n)}p\text{ is bounded}
    \end{aligned}
    \right\}
    $$
    is HAW on $\TT$. In particular, it has full Hausdorff dimension in $\TT$.    
\end{corollary}

\begin{theorem}
    Let $\mathcal C\subset\TT$ be a regular $C^2$ curve satisfying the assumption of Theorem \ref{T:main-curves}. Then the set $$
    \{
    p\in\mathcal{C}: A^+p\text{ is not fiberwise divergent}
    \}
    $$
    is HAW on $\mathcal{C}$. In particular, it has full Hausdorff dimension in $\mathcal{C}$.
\end{theorem}

\subsection{Organization of the paper and strategy of the proof} This paper is organized as follows. In Section \ref{S:Schmidt}, we recall the definitions and properties of Schmidt's game and its variants--the hyperplane absolute game and the hyperplane potential game. In Section \ref{S:danger}, we define the dangerous sets to be avoided for the target set \eqref{E:noULC}, and prove the simplex lemmata that control the distribution of dangerous sets at the same scale. Here we shall need a new simplex lemma for multiple weights (see Lemma \ref{L:simplex0}) which has not previously appeared in this form. 

The most technical part of this paper is Section \ref{SS:potential}. Given a ball $B=B(z,\rho)$, when considering those dangerous sets $\Delta$ with $\Delta\cap B\neq\varnothing$ and with certain level of denominators, we need to group them according to their heights. The resulting sub-families of dangerous sets are $\cD^s$ and $\cD^c_j\,(j\geq 0)$ in Propositions \ref{P:potential1} and \ref{P:potential2}. This grouping is crucial to Alice's collections of affine line neighborhoods at general stages.

In Section \ref{S:win}, we play the hyperplane potential game on $\RR^2$ with the target set \eqref{E:noULC}. We divide all turns of the game into stages, and at each stage, we carefully design Alice's strategy of choosing affine line neighborhoods based on the above grouping. Finally, we prove that the target set \eqref{E:noULC} is hyperplane potential winning, or equivalently, it is hyperplane absolute winning.

In Section \ref{S:further}, we deduce Theorem \ref{T:1+epsilon} from the proof of Theorem \ref{T:main}, in which we use the property that any finite intersection of winning sets is again winning. More explicitly, this trick amounts to placing the finitely many winning strategies into the different residue classes, and showing that putting them together gives a new winning strategy. 

Finally, in Section \ref{S:curves}, we prove Theorems \ref{T:main-curves} and \ref{T:lines} by transferring the
arithmetically controlled deletions of affine line neighborhoods from the ambient potential game to the parameter spaces of $C^2$ curves and lines.

\section{Preliminaries on Schmidt's Games}\label{S:Schmidt}

\subsection{Schmidt's $(\alpha,\beta)$-game}\label{alphabeta}
	
	We first recall Schmidt's $(\alpha,\beta)$-game introduced in \cite{Sch66}. It involves two parameters $\alpha,\beta\in(0,1)$  and is played by two players Alice and Bob on a complete metric space $(X,\dist)$ with a target set $S\subset X$. Bob starts the game by choosing a closed ball $B_0=B(z_0,\rho_0)$ in $X$ with center $z_0$ and radius $\rho_0$. After Bob chooses a closed ball $B_n = B(z_n,{\rho}_n) $, Alice chooses  $A_n = B(z_n', {\rho}_n')$ with $${\rho}'_n=\alpha {\rho}_n\quad\text{and}\quad
	\dist(z'_n, z_n) \leq (1-\alpha){\rho}_n,$$
	and then Bob chooses  $B_{n+1} = B(z_{n+1},{\rho}_{n+1}) $ with $${\rho}_{n+1}=\beta {\rho}'_{n}\quad \text{and}\quad  \dist(z_{n+1}, z'_n) \leq (1-\beta){\rho}'_n,$$
	etc. This implies that the balls are nested:
	$$
	B_0 \supset A_0 \supset B_1 \supset\cdots ;$$
	Alice wins the game if the unique point $\bigcap_{n=0}^\infty A_n=\bigcap_{n=0}^\infty B_n$ belongs to $S$. The set $S$ is \textsl{$(\alpha,\beta)$-winning} if Alice has a winning strategy, is \textsl{$\alpha$-winning} if it is $(\alpha,\beta)$-winning for any $\beta\in(0,1)$, and is \textsl{winning} if it is $\alpha$-winning for some $\alpha\in (0,1)$. 
    
    Regarding this game and its winning sets, Schmidt \cite{Sch66} proved that:	
	\begin{itemize}
		\item[$\bullet$]  winning subsets of Riemannian manifolds are \textit{thick}, that is, their intersection with any nonempty open subset has full Hausdorff dimension;
        \item[$\bullet$]  a countable intersection of $\alpha$-winning sets is again $\alpha$-winning;
        \item[$\bullet$]  if $S$ is $\alpha$-winning and $f:X\to X$ is bi-Lipschitz, then $f(S)$ is $\alpha'$-winning, where $\alpha'$ depends on $\alpha$ and the bi-Lipschitz constant of $f$.
	\end{itemize}
	
	Schmidt's game has been a powerful tool for proving 
	thickness of intersections of certain countable families of sets, see e.g.\ \cite{An13, An16,  BBFKW10, BFK11, BFKRW12, Da1, Da2, Da3}. However, for a fixed $\alpha\in (0,1)$, the class of $\alpha$-winning subsets of a Riemannian manifold depends on the choice of the metric, and is not known to be preserved by diffeomorphisms.

	\subsection{Hyperplane absolute game on $\RR^d$}\label{haw}
	
	Inspired by ideas of McMullen \cite{Mc10}, the hyperplane absolute game on the Euclidean space $\RR^d$ was introduced in \cite{BFKRW12}.
	It has the advantage that the family of its winning sets is preserved by $C^1$ diffeomorphisms.
	Let $S \subset \RR^d$ be a target set and let
	$\beta \in \left(0, \frac13 \right)$.
	As before, Bob begins by choosing a closed ball $B_0$ of radius ${\rho}_0$.
	For an affine hyperplane $L\subset\RR^d$ and $r>0$, we denote the $r$-neighborhood of $L$ by
	$$L^{(r)}:=\{z\in\RR^d:\mathrm{dist}(z,L)\leq r\}.$$
	After Bob chooses a closed ball $B_n$ of radius ${\rho}_n$, Alice chooses a hyperplane neighborhood $L_n^{(r_n)}$ with $r_n\le\beta {\rho}_n$, and then Bob chooses a closed ball $B_{n+1}\subset B_n\setminus L_n^{(r_n)}$ of radius ${\rho}_{n+1}\ge\beta {\rho}_n$. 
    
    We say that Alice wins by default if at any time Bob has no legal move, or $\rho_n \not\to 0$ as $n\to +\infty$. Otherwise, Alice wins the game if and only if
	$$\bigcap_{n=0}^{+\infty} B_n\cap S\ne\emptyset.$$ The set $S$ is \textsl{$\beta$-hyperplane absolute winning} (\textsl{$\beta$-HAW} for short) if Alice has a winning strategy, and is \textsl{hyperplane absolute winning} (\textsl{HAW} for short) if it is $\beta$-HAW for any $\beta\in(0,\frac{1}{3})$.

	\begin{lemma}[\cite{BFKRW12}]\label{L:HAW-Rd}
		\begin{itemize}
			\item[(i)] HAW subsets   are winning, and hence thick.
			\item[(ii)] A countable intersection of HAW subsets   is again HAW.
			\item[(iii)]   The image of an HAW set under a $C^1$ diffeomorphism $\RR^d \to \RR^d$ is HAW.
		\end{itemize}
	\end{lemma}
	
	\subsection{HAW subsets of a manifold}\label{hawmfld}
	
	The notion of HAW sets has been extended to subsets of $C^1$ manifolds in \cite{KW15}.
	This is done in two steps. First one defines the hyperplane absolute game on an open subset $U\subset \RR^d$. It is defined just as the hyperplane absolute game on
	$\RR^d$, except for requiring that Bob's first move $B_0$ be
	contained in $U$. If Alice has a winning strategy, we say that $S$ is {\sl HAW on $U$}.
	Now let $M$ be a $d$-dimensional $C^1$ manifold, and let $\{(U_\alpha, \varphi_\alpha)\}$ be a $C^1$ atlas, that is, $\{U_\alpha\}$ is an open cover of $M$, and each $\varphi_\alpha$ is a $C^1$ diffeomorphism from $U_\alpha$ onto the open subset $\varphi_\alpha(U_\alpha)$ of $\RR^d$. A subset $S\subset M$ is said to be \textsl{HAW on $M$} if for each $\alpha$, $\varphi_\alpha(S\cap U_\alpha)$ is HAW on $\varphi_\alpha(U_\alpha)$. This definition is independent of the choice of the atlas (see \cite[Corollary 3.4]{KW15}).
	
	\begin{lemma}\label{L:HAW-mnfd}
		\begin{itemize}
			\item[(i)]  HAW subsets of a $C^1$ manifold are thick.
			\item[(ii)]  A countable intersection of HAW subsets of a $C^1$ manifold is again HAW.
			\item[(iii)] Let $f:M\to N$ be a diffeomorphism between $C^1$ manifolds, and let $S\subset M$ be an HAW subset of $M$. Then $f(S)$ is an HAW subset of $N$.
            \item[(iv)] The complement of any set with zero Hausdorff dimension is HAW.
		\end{itemize}
	\end{lemma}
	
	\begin{proof}
		(i)--(iii) appeared as \cite[Proposition 3.5]{KW15} and are clear from Lemma \ref{L:HAW-Rd}. 
        (iv) follows from \cite[Proposition 4.6]{BHNS25}, together with the equivalence between the potential and absolute games \cite[Theorem C.8]{FSU18}.
	\end{proof}

    The following statement gives a sufficient condition under which HAW property is transitive.

    
    

\begin{lemma}\label{L:HAW-open}
Let $M$ be a $C^1$ manifold and let $M'\subset M$ be closed. Assume that
$M\setminus M'$ is HAW on $M$. For $S \subseteq M$, $S\cap \left( M\setminus M' \right)$ is HAW on $M\setminus M'$ if and only if $S$ is HAW on $M$.
\end{lemma}

\begin{proof}
Suppose $S$ is HAW on $M$; in particular, Alice has a winning strategy for every Bob's play starting from a ball in $M \setminus M'$, so $S\cap \left( M\setminus M' \right)$ is HAW on $M \setminus M'$.

Conversely, suppose that $S\cap \left( M\setminus M' \right)$ is HAW on $M\setminus M'$.  Since $M'$ is closed, the set $S \cup M'$ is HAW on $M$. Indeed: either one of Bob's balls is contained in $M \setminus M'$, and then Alice applies the strategy for $S\cap \left( M\setminus M' \right)$, or it never happens, and Alice wins since the outcome is in $M'$. Finally, the set $S \cap (M \setminus M')  = \left( S \cup M' \right) \cap (M \setminus M')$, and thus also $S$ as its superset, is HAW on $M$ as an intersection of two HAW sets. 
\end{proof}

	\subsection{Hyperplane potential game}\label{SS:HPW}
	
	Finally, we recall the hyperplane potential game introduced in \cite{FSU18}. It is played on $\RR^d$ and has the same winning sets as the hyperplane absolute game does. This allows one to prove the HAW property of a set $S\subset\RR^d$ by showing that it is winning for the hyperplane potential game (see \cite{AGGL19, AGK15, GW18, GY19, NS14} for example).
	
	Let $S\subset\RR^d$ be a target set, and let $\beta\in(0,1)$, $\gamma>0$. The \textsl{$(\beta,\gamma)$-hyperplane potential game} is defined as follows: Bob begins by choosing a closed ball $B_0\subset\RR^d$. After Bob chooses a closed ball $B_n$ of radius ${\rho}_n$, Alice chooses an at most countable indexed family of hyperplane neighborhoods $\{L_{n,i}^{(r_{n,i})} : i\in I_n\}$ such that
	\begin{equation}\label{E:potential_game_ineq}
    \sum_{i\in I_n} r_{n,i}^\gamma\le(\beta {\rho}_n)^\gamma,
    \end{equation}
	and then Bob chooses a closed ball $B_{n+1}\subset B_n$ of radius $\rho_{n+1}\ge\beta \rho_n$. 
    
    We say that Alice wins by default if $$
    \rho_n\not\to 0\text{ as }n\to +\infty,\quad\text{or}\quad \bigcap_{n=0}^\infty B_n\cap\bigcup_{n=0}^\infty\bigcup_{i\in{I_n}} L_{n,i}^{(r_{n,i})}\ne\emptyset.
    $$
    Otherwise, Alice wins the game if and only if
	$$\bigcap_{n=0}^\infty B_n\cap S\ne\emptyset.$$
	The set $S$ is \textsl{$(\beta,\gamma)$-hyperplane potential winning} (\textsl{$(\beta,\gamma)$-HPW} for short) if Alice has a winning
	strategy, and is \textsl{hyperplane potential winning} (\textsl{HPW} for short) if it is $(\beta,\gamma)$-HPW for any $\beta\in(0,1)$ and $\gamma>0$. The following lemma is a special case of \cite[Theorem C.8]{FSU18}.
	
	\begin{lemma}\label{L:HAWHPW}
		A subset $S\subset\RR^d$ is HPW if and only if it is HAW.
	\end{lemma}

\section{The dangerous sets and simplex lemmata}\label{S:danger}

In this section, we define the dangerous sets to be avoided for the target set \eqref{E:noULC}. Then we prove the simplex lemmata, which state that: locally, all these dangerous sets of the same type and at the same scale must be contained in at most two common line neighborhoods. 

\subsection{Definitions of dangerous sets} We first introduce some notations. Write $$
\cP:=\left\{P=\left(\frac{p}{q},\frac{r}{q}\right):(p,r,q)\in\ZZ^2\times \NN,\;\gcd(p,r,q)=1\right\}.
$$
For $P\in\cP$, $s\in [0,1]$, and $c>0$, we write \begin{equation*}
    \begin{aligned}
        \Delta_c(P;s)&:=\left\{(x,y)\in\RR^2:
        \left|x-\frac{p}{q}\right|<\frac{c}{q^{1+s}},\;
        \left|y-\frac{r}{q}\right|<\frac{c}{q^{2-s}}
        \right\};\\
        \Delta_c^1\left(P\right)&:=\left\{(x,y)\in\RR^2:
        \left|x-\frac{p}{q}\right|<\frac{c}{q^2}
        \right\};\\
        \Delta_c^2\left(P\right)&:=\left\{(x,y)\in\RR^2:
        \left|y-\frac{r}{q}\right|<\frac{c}{q^2}
        \right\}.
    \end{aligned}
\end{equation*}
Throughout, a dangerous set is regarded as a labeled object together with its defining parameters. Thus two dangerous sets having the same underlying subset of $\RR^2$ but different defining parameters are regarded as distinct.

We begin with a simple observation on the set of multiplicatively well approximable vectors with respect to a rational point $P\in\cP$.

\begin{lemma}\label{L:WcP}
    Let $P=(\frac{p}{q},\frac{r}{q})\in\cP$ and $c>0$. Then the set $$
    W_c(P):=\left\{
    (x,y)\in\RR^2:q|qx-p||qy-r|<c^2
    \right\}
    $$
    is contained in the union $$
    \bigcup_{s\in [0,1]}\Delta_c(P;s)\cup \Delta^1_c(P)\cup \Delta^2_c(P).
    $$
\end{lemma}
\begin{proof}
Let $(x,y)\in\RR^2$ satisfy the inequality \begin{equation}\label{E:WcP}
q|qx-p||qy-r|<c^2.
\end{equation} 
If $|qx-p|=0$, then it is clear that $(x,y)\in\Delta^1_c(P)$. If $|qx-p|\geq c$, then it follows from \eqref{E:WcP} that $q|qy-r|<c$, that is, $(x,y)\in\Delta_c^2(P)$. So we may assume that $0<|qx-p|<c$. By symmetry, we also assume that $0<|qy-r|<c$.

When $q=1$, it follows that for any $s\in [0,1]$ one has $$
0<q^s|qx-p|<c\quad \text{and}\quad 0<q^{1-s}|qy-r|<c,
$$
that is, $(x,y)\in \Delta_c(P;s)$. When $q\geq 2$, we see from \eqref{E:WcP} that $$
1<\frac{\log(\frac{c}{|qx-p|})}{\log{q}}+\frac{\log(\frac{c}{|qy-r|})}{\log{q}}.
$$
Since the two terms on the right-hand side of the above inequality are positive, we may choose some $s\in [0,1]$ such that $$
s<\frac{\log(\frac{c}{|qx-p|})}{\log{q}}\quad \text{and}\quad 1-s<\frac{\log(\frac{c}{|qy-r|})}{\log{q}}.
$$
This simply means that $(x,y)\in \Delta_c(P;s)$. The proof is complete.
\end{proof}

Using Lemma \ref{L:WcP}, we may find a subset of our target set \eqref{E:noULC} that is easier to analyze. For $P=\left(\frac{p}{q},\frac{r}{q}\right)\in\cP$, $Q\geq q$, and $\eta>0$, if we set \begin{equation}\label{E:choosec}
    c=c(P,Q,\eta)=\eta\sqrt{\frac{q}{Q}},
\end{equation}
then it is clear that \begin{equation}
    W_c(P)=\left\{
    (x,y)\in\RR^2:Q|qx-p||qy-r|<\eta^2
    \right\}.
\end{equation}
In particular, we have \begin{equation}\label{E:unionWcP}
    \bigcup_{P\in\cP,q\leq Q}W_c(P)=\left\{
    (x,y)\in\RR^2:Q\cdot \min_{1\leq q\leq Q}\la qx\ra\la qy\ra<\eta^2
    \right\}.
\end{equation}
The following corollary is immediate: 

\begin{corollary}\label{C:Eeta}
Let $(Q_m)_{m\geq 0}$ be a strictly increasing sequence of positive numbers with $Q_0>1$ and $Q_m\to\infty$. Then for any $\eta>0$, the sets 
$\RR^2\setminus E\subseteq\RR^2\setminus E_{\bQ}$ are contained in the union of \begin{equation}\label{E:Eeta0}
    \left\{
    (x,y)\in\RR^2:Q_0\cdot \min_{1\leq q\leq Q_0}\la qx\ra\la qy\ra<\eta^2
    \right\}
\end{equation}
and
\begin{equation}\label{E:Eeta1}
    \bigcup_{m\geq 1}\bigcup_{\substack{P\in\cP,\\ Q_{m-1}<q\leq Q_m}}
    \Big(\bigcup_{s\in [0,1]}\Delta_{c(P,Q_m,\eta)}(P;s)\cup \Delta^1_{c(P,Q_m,\eta)}(P)\cup \Delta^2_{c(P,Q_m,\eta)}(P)\Big).
\end{equation}
\end{corollary}
\begin{proof}
It is clear from \eqref{E:noULC} and \eqref{E:EQ} that for any $\eta>0$, $$
\RR^2\setminus E\subseteq \RR^2\setminus E_{\bQ}\subseteq \bigcup_{m\geq 0}\left\{
    (x,y)\in\RR^2:Q_m\cdot \min_{1\leq q\leq Q_m}\la qx\ra\la qy\ra<\eta^2
    \right\}.
$$
By \eqref{E:choosec} and \eqref{E:unionWcP}, we see that the above right-hand side equals\begin{equation}\label{E:changeW}
\begin{aligned}
\bigcup_{m\geq 0}\bigcup_{P\in \cP,q\leq Q_m}W_{c(P,Q_m,\eta)}(P)
={}&\bigcup_{\substack{P\in\cP,\\q\leq Q_0}}W_{c(P,Q_0,\eta)}(P)\\
&\cup\bigcup_{m\geq 1}\bigcup_{\substack{P\in\cP,\\Q_{m-1}<q\leq Q_m}}W_{c(P,Q_m,\eta)}(P).
\end{aligned}
\end{equation}
Here $c(P,Q,\eta)$ decreases as $Q$ increases, so for a fixed
denominator, the first scale at which it occurs contains all its later contributions.
For the right-hand side of \eqref{E:changeW}: the term $m=0$ is \eqref{E:Eeta0}; by Lemma \ref{L:WcP}, the union of other terms $m\geq 1$ is contained in \eqref{E:Eeta1}. The proof is complete.
\end{proof}

\subsection{Simplex lemmata} In view of Corollary \ref{C:Eeta}, for the sake of convenience, we refer to dangerous sets of the forms \begin{equation}\label{E:type}
\Delta_c(P;s),\quad \Delta^1_c(P),\quad \Delta^2_c(P)
\end{equation}
as \textsl{Type 0}, \textsl{Type 1}, \textsl{Type 2}-dangerous sets, respectively. Moreover, we introduce the \textit{scales} and the simplex lemmata for these three types of dangerous sets as follows. 

\subsubsection{The scale and multi-weight simplex lemma for Type 0-dangerous sets}
We first recall a general result for rational points.

\begin{lemma}\label{L:attachline}
    To each $P=(\frac{p}{q},\frac{r}{q})\in\cP$ and $s\in [0,1]$, one may attach a rational line $L(P;s): A(P;s)x+B(P;s)y+C(P;s)=0$ passing through $P$ such that \begin{equation}\label{E:ABC}
    \begin{gathered}
    A(P;s),B(P;s),C(P;s)\in\ZZ,\qquad
    \gcd(A(P;s),B(P;s),C(P;s))=1,\\
    |A(P;s)|\leq q^s,\qquad |B(P;s)|\leq q^{1-s}.
    \end{gathered}
    \end{equation}
\end{lemma}
\begin{proof}
    See \cite[Lemma 3.1]{An16}. 
\end{proof}

For each $P=(\frac{p}{q},\frac{r}{q})\in\cP$ and $s\in [0,1]$, we choose once and for all a rational line $L(P;s)$ as in Lemma \ref{L:attachline}. Then we define the associated \textit{height} as \begin{equation}\label{E:height}
    H(P;s):=q\cdot \max\{|A(P;s)|,|B(P;s)|\}.
\end{equation}
It follows from \eqref{E:ABC} that \begin{equation}\label{E:Hrange}
    q\leq H(P;s)\leq q^{1+\max\{s,1-s\}}.
\end{equation}
The scale of the dangerous set $\Delta_c(P;s)$ is determined by the denominator $q$ and the height $H(P;s)$.

The next result is the multi-weight simplex lemma for Type 0-dangerous sets. It basically says that all $\Delta_c(P;s)$ that intersect a ball $B$ and have suitable denominators and heights must be contained in the union of two common line neighborhoods. Readers may compare this result with \cite[Lemma 5.5]{GW26}.

\begin{lemma}\label{L:simplex0}
Let $R\geq 2$, $Q,H\geq 1$, $c_*>0$, and $B=B(z,\rho)$. Suppose that there is an indexed family $$
\{(P_i,s_i,c_i):i\in I\},\quad \text{where}\quad P_i\in\cP,s_i\in [0,1],c_i>0,
$$
such that for each $i\in I$, $$
R^{-1}Q<q_i\leq Q,\quad H\leq H(P_i;s_i)<RH,\quad 0<c_i\leq c_*, \quad\text{and}\quad \Delta_{c_i}(P_i;s_i)\cap B\neq \varnothing.
$$ 

If \begin{equation}\label{E:condition0}
    4R^2(c_*+\rho H)<1,
\end{equation}
then all $P_i\,(i\in I)$ with $s_i\in [0,\frac{1}{2}]$ lie on one rational line, and all $P_i\,(i\in I)$ with $s_i\in [\frac{1}{2},1]$ lie on another rational line. Consequently, all corresponding $\Delta_{c_i}(P_i;s_i)$ are contained in the $\frac{3c_*R^3}{2H}$-neighborhood of these two lines. 
\end{lemma}
\begin{remark}
    For each nonempty half-family, its common line may be chosen to be an actual attached line $L(P_{i_*};s_{i_*}):A_*x+B_*y+C_*=0$ from that half-family.  If $q_*$ is the denominator of $P_{i_*}$ and $\xi_*=\max\{|A_*|,|B_*|\}$, then
\begin{equation}\label{E:simplex-label}
 |A_*\cdot B_*|\leq q_*,
 \qquad H\leq q_*\xi_*<RH.
\end{equation}
\end{remark}

\begin{proof}
By symmetry, we only deal with the subfamily $\{(P_i,s_i,c_i):i\in I,s_i\in [0,\frac{1}{2}]\}$. If this subfamily is empty, then there is nothing to prove. Otherwise, we put $$
s_*:=\inf_{i\in I}s_i\in \left[0,\frac{1}{2}\right].
$$
Then there exists a sequence $(i_k)_{k\geq 1}$ in $I$ with $s_{i_k}\in [0,\frac{1}{2}]$ and $s_{i_k}\to s_*$. If $s_*$ is attained, we simply take the approximating sequence to be $s_{i_k}\equiv s_*$.

Since there are only finitely many denominators $R^{-1}Q<q_{i_k}\leq Q$ and finitely many numerators $(p_{i_k},r_{i_k})$ by the condition $$
\Delta_{c_{i_k}}(P_{i_k};s_{i_k})\cap B(z,\rho)\neq \varnothing, \quad \text{where}\quad z=(z_1,z_2),
$$
that is,
$$
\left|\frac{p_{i_k}}{q_{i_k}}-z_1\right|\leq \frac{c_{i_k}}{q_{i_k}^{1+s_{i_k}}}+\rho\leq c_*+\rho,\quad\text{and}\quad \left|\frac{r_{i_k}}{q_{i_k}}-z_2\right|\leq \frac{c_{i_k}}{q_{i_k}^{2-s_{i_k}}}+\rho\leq c_*+\rho,
$$
we see that along this sequence only finitely many centers $P_{i_k}$ can occur. After passing to a subsequence, we may assume that $P_{i_k}=P_*$ for all $k\geq 1$. Write $P_*=(\frac{p_*}{q_*},\frac{r_*}{q_*})$. Moreover, since the attached lines $L(P_*;s_{i_k})$ are determined by the integer triples $(A(P_*;s_{i_k}),B(P_*;s_{i_k}),C(P_*;s_{i_k}))$ where $$
|A(P_*;s_{i_k})|\leq q_{i_k}^{s_{i_k}}\leq Q^{1/2},\quad 
|B(P_*;s_{i_k})|\leq q_{i_k}^{1-s_{i_k}}\leq Q,\quad 
$$
and $$
C(P_*;s_{i_k})=-\frac{A(P_*{;}s_{i_k})p_*+B(P_*{;}s_{i_k})r_*}{q_*},
$$
there are only finitely many lines attached to the point $P_*$. After passing to a further subsequence, we may assume that $$(A(P_*;s_{i_k}),B(P_*;s_{i_k}),C(P_*;s_{i_k}))=(A_*,B_*,C_*)\text{ for all }k\geq 1.$$
    We write $L_*:A_*x+B_*y+C_*=0$. Since this triple occurs along the subsequence, fix one of its indices and call it $i_*$; thus $L_*=L(P_{i_*};s_{i_*})$.
    We claim that: \begin{equation}\label{E:sameline}
    \text{all }P_i\,(i\in I)\text{ with }s_i\in \left[0,\frac{1}{2}\right] \text{ lie on }L_*.
\end{equation}

In fact, for each $i\in I$ with $s_i\in [0,\frac{1}{2}]$, we consider the quantity $$
N_i:=A_*p_i+B_*r_i+C_*q_i\in\ZZ.
$$ 
To show that $P_i\in L_*$, it suffices to show that $N_i=0$, or equivalently, $|N_i|<1$. Since $P_*\in L_*$, we have \begin{equation}\label{E:Ni}
N_i=q_i\left(
A_*\left(\frac{p_i}{q_i}-\frac{p_*}{q_*}\right)+B_*\left(\frac{r_i}{q_i}-\frac{r_*}{q_*}\right)
\right).
\end{equation}
Now we estimate the right-hand side of \eqref{E:Ni}.

We fix any $i\in I$ with $s_i\in [0,\frac{1}{2}]$, and choose some $i_k$ with $s_{i_k}\leq s_i$. Note that 
\begin{subequations}
\begin{gather}
    \frac{|A_*|}{q_i^{s_i}}\overset{\eqref{E:ABC}}{\leq} \frac{q_*^{s_{i_k}}}{q_i^{s_i}}\leq \left(\frac{q_*}{q_i}\right)^{s_i}<R^{1/2}; \qquad 
    \frac{q_i|A_*|}{q_*^{1+s_{i_k}}}\overset{\eqref{E:ABC}}{\leq} \frac{q_i}{q_*}<R; \label{E:calA}\\
    \frac{|B_*|}{q_i^{1-s_i}}\overset{\eqref{E:height}}{\leq} \frac{H(P_*;s_{i_k})}{q_*q_i^{1-s_i}}<\frac{RH(P_i;s_i)}{q_*q_i^{1-s_i}}
    \overset{\eqref{E:Hrange}}{\leq} \frac{Rq_i}{q_*}<R^2;
    \qquad
    \frac{q_i|B_*|}{q_*^{2-s_{i_k}}}\overset{\eqref{E:ABC}}{\leq} \frac{q_i}{q_*}<R. \label{E:calB}
\end{gather}
\end{subequations}
Moreover, \begin{equation}\label{E:calAB}
    q_i\rho(|A_*|+|B_*|)\overset{\eqref{E:height}}{\leq} \frac{2q_i\rho H(P_*;s_{i_k})}{q_*}<2R^2\rho H.
\end{equation}
Since both $\Delta_{c_i}(P_i;s_i)$ and $\Delta_{c_{i_k}}(P_*;s_{i_k})$ intersect $B(z,\rho)$, we have $$
\left|\frac{p_i}{q_i}-\frac{p_*}{q_*}\right|\leq \frac{c_i}{q_i^{1+s_i}}+\frac{c_{i_k}}{q_*^{1+s_{i_k}}}+2\rho,
\quad \text{and}\quad 
\left|\frac{r_i}{q_i}-\frac{r_*}{q_*}\right|\leq \frac{c_i}{q_i^{2-s_i}}+\frac{c_{i_k}}{q_*^{2-s_{i_k}}}+2\rho.
$$
It follows that 
    \begin{align*}
        |N_i|&\leq q_i\left(
        |A_*|\left|\frac{p_i}{q_i}-\frac{p_*}{q_*}\right|
        +|B_*|\left|\frac{r_i}{q_i}-\frac{r_*}{q_*}\right|
        \right)\tag{by \eqref{E:Ni}}\\
        &\leq q_i\left(
        c_*\left(\frac{|A_*|}{q_i^{1+s_i}}+\frac{|A_*|}{q_*^{1+s_{i_k}}}+\frac{|B_*|}{q_i^{2-s_i}}+\frac{|B_*|}{q_*^{2-s_{i_k}}}\right)
        +2\rho(|A_*|+|B_*|)
        \right)\\
        &\leq c_*\left(R^{1/2}+R+R^2+R\right)+4R^2\rho H
        \tag{by \eqref{E:calA}\eqref{E:calB}\eqref{E:calAB}}
        \\
        &<1. \tag{by \eqref{E:condition0}}
    \end{align*}
This completes the proof of the claim \eqref{E:sameline}.

    Finally, let $(x,y)\in \Delta_{c_i}(P_i;s_i)$. Then we see that \begin{align*}
        |A_*x+B_*y+C_*|&=\left|A_*\left(x-\frac{p_i}{q_i}\right)+B_*\left(y-\frac{r_i}{q_i}\right)\right|
        \tag{by \eqref{E:sameline}}\\
        &\leq c_*\left(\frac{|A_*|}{q_i^{1+s_i}}+\frac{|B_*|}{q_i^{2-s_i}}\right)
        \\
        &\leq \frac{c_*}{q_i}(R^{1/2}+R^2) \tag{by \eqref{E:calA}\eqref{E:calB}}.
    \end{align*}
    Moreover, $$
    \sqrt{A_*^2+B_*^2}\overset{\eqref{E:height}}{\geq} \frac{H(P_*;s_{i_k})}{q_*}>\frac{H}{Rq_i}.
    $$
    So we conclude that the distance of $(x,y)$ to the line $L_*$ is $$
    \frac{|A_*x+B_*y+C_*|}{\sqrt{A_*^2+B_*^2}}<\frac{c_*R(R^{1/2}+R^2)}{H}\leq \frac{3c_*R^3}{2H}.
    $$
    This completes the proof.
\end{proof}

\subsubsection{The scale and simplex lemma for Type 1 and Type 2-dangerous sets} By symmetry, we only deal with Type 1-dangerous sets. For each $P=(\frac{p}{q},\frac{r}{q})\in\cP$, we define its height as $H(P):=q^2$. The scale of the dangerous set $\Delta^1_c(P)$ is determined by the height $H(P)$. The next result is the one-dimensional simplex lemma for Type 1-dangerous sets.

\begin{lemma}\label{L:simplex1}
Let $R\geq 2$, $H\geq 1$, $c_*>0$, and $B=B(z,\rho)$. Suppose that there is an indexed family $$
\{(P_i,c_i):i\in I\},\quad \text{where}\quad P_i\in\cP,c_i>0,
$$
such that for each $i\in I$, $$
H\leq H(P_i)<RH,\quad 0<c_i\leq c_*,\quad \text{and}\quad \Delta^1_{c_i}(P_i)\cap B\neq\varnothing.
$$
If \begin{equation}\label{E:condition1}
    2R(c_*+\rho H)<1,
\end{equation}
then all $\frac{p_i}{q_i}\,(i\in I)$ are equal. Consequently, all $\Delta^1_{c_i}(P_i)$ are contained in a vertical line neighborhood of radius $\frac{c_*}{H}$.
\end{lemma}
\begin{proof}
    For any $i,j\in I$,  since both $\Delta^1_{c_i}(P_i)$ and $\Delta^1_{c_j}(P_j)$ intersect $B(z,\rho)$, we have $$
    \left|\frac{p_i}{q_i}-\frac{p_j}{q_j}\right|\leq \frac{c_i}{q_i^2}+\frac{c_j}{q_j^2}+2\rho\leq \frac{2(c_*+\rho H)}{H}
    \overset{\eqref{E:condition1}}{<}\frac{1}{RH}<\frac{1}{q_iq_j}.
    $$
    So $|p_iq_j-p_jq_i|<1$ and $p_iq_j-p_jq_i\in\ZZ$ together imply that $\frac{p_i}{q_i}=\frac{p_j}{q_j}$. Finally, all $\Delta^1_{c_i}(P_i)$ are contained in a neighborhood of the common vertical line $L^1:x=\frac{p_i}{q_i}$ with radius $\frac{c_i}{q_i^2}\leq \frac{c_*}{H}$.
\end{proof}

\section{Potential Estimates}\label{SS:potential} In this section, we define and calculate the \textit{total potential} of line neighborhoods that are chosen by simplex lemmata for several scales. This will be crucial to Alice's strategy in the next section.

We fix two parameters $\beta\in (0,1)$ and $\gamma>0$. For an affine line $L\subseteq \RR^2$ and $r>0$, we write $$
L^{(r)}:=\{z\in\RR^2:\dist(z,L)\leq r\}.
$$
Whenever $L$ is rational, we choose its primitive integral equation
$Ax+By+C=0$ and write
\begin{equation}\label{E:line-height}
\xi(L):=\max\{|A|,|B|\}.
\end{equation}
For a countable collection of affine line neighborhoods $\cL:=\{L_j^{(r_j)}:j\in J\}$, we define its \textit{total $\gamma$-potential} as \begin{equation}\label{E:potential}
    \varphi(\cL):=\sum_{j\in J}r_j^{\gamma}.
\end{equation}
We choose an integer $R>\max\{4,\beta^{-1}\}$ sufficiently large so that \begin{equation}\label{E:Rcondition}
    2^{1-\gamma}3^{1+\gamma}R^{-\gamma}\leq \beta^{2\gamma}(1-R^{-\gamma})(1-R^{-\gamma/2}),
\end{equation}
and choose $\eta>0$ sufficiently small so that \begin{equation}\label{E:eta}
    0<\eta<\frac{1}{4R^2(1+R^5)}.
\end{equation}
We also write $\theta=\eta R^4$.

Before proceeding to our key propositions, we first summarize some notations for three types of dangerous sets. Let $\widetilde{\cD}$ denote the collection of all dangerous sets as in \eqref{E:type}. For any $\Delta\in\widetilde{\cD}$ as a labeled object together with its defining parameters, we write \begin{align*}
q(\Delta)&:=\text{the denominator of }P;\\
c(\Delta)&:=\text{the parameter }c;\\
H(\Delta)&:={
\begin{cases}
    H(P;s), &\text{if }\Delta=\Delta_c(P;s)\\
    H(P), &\text{if }\Delta=\Delta^1_c(P)\text{ or }\Delta^2_c(P)
\end{cases};
}
\\
j(\Delta)&:=\left\lfloor \frac{\log{H(\Delta)}}{\log{R}}\right\rfloor.
\end{align*}
Given a ball $B=B(z,\rho)$, when considering those dangerous sets $\Delta$ with $\Delta\cap B\neq\varnothing$ and with certain level of denominators, we will group them according to their heights: \begin{enumerate}
    \item[(i)] $\rho\cdot R^{j(\Delta)}<\theta R$. In this case, we just consider the ball $B=B(z,\rho)$, and study those dangerous sets with smaller heights;
    \item[(ii)] $\rho\cdot R^{j(\Delta)}\geq \theta R$. In this case, since $R>\beta^{-1}$, we have $\beta\rho>\theta R^{-j(\Delta)}$, so there exists a unique $n\ge 1$ with 
    $\beta^n\rho>\theta R^{-j(\Delta)}\geq \beta^{n+1}\rho$. We will consider another ball $B'=B(z',\rho')$ where $\rho'=\beta^{n+1}\rho$, and study the case $$
    \theta R^{-j(\Delta)}\geq \rho'>\beta\cdot\theta R^{-j(\Delta)}.
    $$
\end{enumerate}

\subsection{Total potential of affine line neighborhoods with small heights}
Given $R\leq Q_-<Q$ and $B=B(z,\rho)$, we consider the following family of dangerous sets \begin{equation}\label{E:defD}
\cD^{s}(Q_-,Q;B):=\left\{
\Delta\in\widetilde{\cD}\;\middle|\;
\begin{aligned}
&Q_-<q(\Delta)\leq Q,\quad c(\Delta)=c(P,Q,\eta),\\
&\rho\cdot R^{j(\Delta)}<\theta R,\quad \Delta\cap B(z,\rho)\neq \varnothing
\end{aligned}
\right\}.
\end{equation}
The first potential estimate concerns those affine line neighborhoods covering all dangerous sets in $\cD^{s}(Q_-,Q;B)$.

\begin{proposition}\label{P:potential1}
Let $R\leq Q_-<Q$ and $B=B(z,\rho)$. Suppose that \begin{equation}\label{E:QQ-condition}
    Q\geq \frac{\theta^2R^4}{\rho^2Q_-}.
\end{equation}
Then there exists a finite family $\cL^{s}(Q_-,Q;B)$ of affine line neighborhoods so that $$
\bigcup_{\Delta\in \cD^{s}(Q_-,Q;B)}\Delta\subseteq \bigcup_{L^{(r)}\in \cL^{s}(Q_-,Q;B)}L^{(r)},
\qquad 
\text{and} \qquad
\varphi(\cL^{s}(Q_-,Q;B))<(\beta\rho)^{\gamma}.
$$
The family can be chosen so that every $L^{(r)}$ in it has rational
central line and
\begin{equation}\label{E:potential1-control}
r \xi(L)|A\cdot B|<\frac32\eta R^4,
\end{equation}
where $(A,B,C)$ is the primitive triple defining $L$.
\end{proposition}
\begin{proof}
Write $\cD^s=\cD^{s}(Q_-,Q;B)$ and 
$$
\cD^{s,i}:=\{\text{Type }i\text{-dangerous sets in }\cD^s\},\quad i\in \{0,1,2\}.
$$
Let us first consider $\cD^{s,0}$. For any pair $(k,j)\in \NN_0^2$ with \begin{equation}\label{E:kjrange}
k<\frac{\log(Q/Q_-)}{\log{R}}\quad\text{and}\quad j>\frac{\log{Q}}{\log{R}}-(k+2),
\end{equation}
we define \begin{equation}
    \cD^{s,0}_{k,j}:=\left\{\Delta\in \cD^{s,0}: R^{-(k+1)}Q<q(\Delta)\leq R^{-k}Q,\;j(\Delta)=j
    \right\}.
\end{equation}
We observe that: \begin{equation}\label{E:D0par}
    \cD^{s,0}=\bigsqcup_{\substack{(k,j)\in \NN_0^2\\ \text{with \eqref{E:kjrange}}}}\cD_{k,j}^{s,0}.
\end{equation}

In fact, for $\Delta\in\cD^{s,0}$, since $q(\Delta)\leq Q$, we may find a unique $k\geq 0$ such that 
$$R^{-(k+1)}Q<q(\Delta)\leq R^{-k}Q.$$
In particular, since $q(\Delta)>Q_-$, we have $k<\frac{\log(Q/Q_-)}{\log{R}}$.
Moreover, we see that $$
j(\Delta)>\frac{\log{H(\Delta)}}{\log{R}}-1\geq \frac{\log{q(\Delta)}}{\log{R}}-1>\frac{\log{Q}}{\log{R}}-(k+2),
$$
so there exists a unique $j>\frac{\log{Q}}{\log{R}}-(k+2)$ such that $j(\Delta)=j$. This verifies \eqref{E:D0par}.

Note that for any $\Delta\in \cD^{s,0}_{k,j}$, $$
4R^2(c(\Delta)+\rho R^{j(\Delta)})\overset{\eqref{E:defD}}{<}4R^2(\eta R^{-k/2}+\theta R)\overset{\eqref{E:eta}}{<}1.
$$
So applying Lemma \ref{L:simplex0} to each $\cD^{s,0}_{k,j}$ gives two affine line neighborhoods of radius $$
r_{k,j}:=\frac{3\eta R^{3-k/2}}{2R^j},
$$ 
whose union covers all dangerous sets in $\cD^{s,0}_{k,j}$. For each nonempty half-family, Lemma \ref{L:simplex0} supplies an attached line $L_*:A_*x+B_*y+C_*=0$.  With its label $q_*$ and $\xi_*$,
\eqref{E:simplex-label} gives $|A_*\cdot B_*|\leq q_*$ and $q_*\xi_*<R^{j+1}$; hence
\[
r_{k,j}\xi_*|A_*\cdot B_*|\leq r_{k,j}\xi_* q_*<\frac32\eta R^{4-k/2}\leq\frac32\eta R^4.
\]
In view of \eqref{E:D0par}, we conclude that the total potential of affine line neighborhoods to cover all dangerous sets in $\cD^{s,0}$ is \begin{equation*}
    \begin{aligned}
        \sum_{\substack{(k,j)\in \NN_0^2\\ \text{with \eqref{E:kjrange}}}}2r_{k,j}^{\gamma}
        &=2^{1-\gamma}(3\eta R^3)^{\gamma}\sum_{0\leq k<\frac{\log(Q/Q_-)}{\log{R}}}R^{-k\gamma/2}\sum_{j>\frac{\log{Q}}{\log{R}}-(k+2)}R^{-j\gamma}\\
        &\leq \frac{2^{1-\gamma}(3\eta R^5Q^{-1})^{\gamma}}{1-R^{-\gamma}}\sum_{0\leq k<\frac{\log(Q/Q_-)}{\log{R}}}R^{k\gamma/2}\\
        &<\frac{2^{1-\gamma}(3\theta R(QQ_-)^{-1/2})^{\gamma}}{(1-R^{-\gamma})(1-R^{-\gamma/2})}.
    \end{aligned}
\end{equation*}
By \eqref{E:QQ-condition} and \eqref{E:Rcondition}, we see that the above right-hand side is not exceeding $\frac{1}{3}(\beta\rho)^{\gamma}$.

Now we consider the $\cD^{s,1}$ case. The $\cD^{s,2}$ case is analogous. For any $j\in \NN_0$ with $j>\frac{2\log{Q_-}}{\log{R}}-1$,
we define \begin{equation}
    \cD^{s,1}_j:=\{\Delta\in\cD^{s,1}:j(\Delta)=j\}.
\end{equation}
Since for any $\Delta\in\cD^{s,1}$, we have $$
j(\Delta)>\frac{\log{H(\Delta)}}{\log{R}}-1=\frac{2\log{q(\Delta)}}{\log{R}}-1>\frac{2\log{Q_-}}{\log{R}}-1,
$$
there exists a unique $j>\frac{2\log{Q_-}}{\log{R}}-1$ such that $j(\Delta)=j$. So we observe that \begin{equation}\label{E:D1par}
    \cD^{s,1}=\bigsqcup_{j>\frac{2\log{Q_-}}{\log{R}}-1}\cD^{s,1}_j.
\end{equation}

Note that for any $\Delta\in\cD^{s,1}_j$,
put
\[
 c_{*,j}:=\eta\min\left\{1,
 \sqrt{\frac{R^{(j+1)/2}}{Q}}\right\}.
\]
Since $j(\Delta)=j$ gives $q(\Delta)<R^{(j+1)/2}$, while
$q(\Delta)\leq Q$, one has $c(\Delta)\leq c_{*,j}$.  Moreover,
$$
2R(c_{*,j}+\rho R^{j(\Delta)})\overset{\eqref{E:defD}}{<}2R(\eta+\theta R)\overset{\eqref{E:eta}}{<}1.
$$
So applying Lemma \ref{L:simplex1} to each $\cD^{s,1}_j$ gives a vertical line neighborhood of radius 
$$
r_j:=\frac{\eta}{R^j}\cdot\sqrt{\frac{R^{\frac{j+1}{2}}}{Q}},
$$
which covers all dangerous sets in $\cD^{s,1}_j$. In view of \eqref{E:D1par}, we conclude that the total potential of affine line neighborhoods to cover all dangerous sets in $\cD^{s,1}$ is \begin{equation*}
    \sum_{j>\frac{2\log{Q_-}}{\log{R}}-1}r_j^{\gamma}\leq \frac{(\eta RQ^{-1/2}Q_-^{-3/2})^{\gamma}}{1-R^{-3\gamma/4}}
    \leq \frac{(\theta R^{-4}(QQ_-)^{-1/2})^{\gamma}}{1-R^{-3\gamma/4}}.
\end{equation*}
By \eqref{E:QQ-condition} and \eqref{E:Rcondition}, we see that the above right-hand side is not exceeding $\frac{1}{3}(\beta\rho)^{\gamma}$.

The resulting family is finite.  Indeed, by \eqref{E:defD}, the indices in \eqref{E:D0par} and \eqref{E:D1par} satisfy $\rho R^j<\theta R$, which gives an upper bound for $j$; the lower bounds above and the finite range of $k$ then leave only finitely many subcollections.

Finally, putting all these affine line neighborhoods for $\cD^s=\cD^{s,0}\cup\cD^{s,1}\cup\cD^{s,2}$ together gives the desired conclusion.
\end{proof}

\subsection{Total potential of affine line neighborhoods with certain height} Given $R\leq Q_-<Q$ and $B'=B(z',\rho')$, we consider the following family of dangerous sets \begin{equation}\label{E:defF}
\cD^c(Q_-,Q;B'):=\left\{
\Delta\in\widetilde{\cD}\;\middle|\;
\begin{aligned}
&Q_-<q(\Delta)\leq Q,\quad c(\Delta)=c(P,Q,\eta),\\
&\theta R^{-j(\Delta)}\geq \rho'>\beta\cdot\theta R^{-j(\Delta)},\quad \Delta\cap B'\neq \varnothing
\end{aligned}
\right\}.
\end{equation}
The second potential estimate concerns those affine line neighborhoods covering all dangerous sets in $\cD^c(Q_-,Q;B')$.

\begin{proposition}\label{P:potential2}
    Let $R\leq Q_-<Q$, $B'=B(z',\rho')$, and write $\cD^c=\cD^c(Q_-,Q;B')$. For $j\geq 0$, consider the subfamily \begin{equation}
        \cD^c_j:=\{\Delta\in\cD^c:j(\Delta)=j\},
    \end{equation}
and suppose that \begin{equation}\label{E:jcondition}
    4R^2(\eta+\rho'R^j)<1.
\end{equation}
Then there exists a finite family $\cL^c_j=\cL^c_j(Q_-,Q;B')$ of affine line neighborhoods such that $$
\bigcup_{\Delta\in \cD^c_j}\Delta\subseteq\bigcup_{L^{(r)}\in \cL^c_j}L^{(r)},
\qquad
\text{and} \qquad
\varphi(\cL^c_j)\leq (\beta\rho')^{\gamma}.
$$
The family can be chosen so that every $L^{(r)}$ in it has rational
central line and
\begin{equation}\label{E:potential2-control}
r \xi(L)|A\cdot B|<\frac32\eta R^4,
\end{equation}
where $(A,B,C)$ is the primitive triple defining $L$.
\end{proposition}
\begin{proof}
Fix $j\geq 0$. For $i\in \{0,1,2\}$, write $$
\cD_j^{c,i}:=\{\text{Type }i\text{-dangerous sets in }\cD^c_j\}.$$
Let us first consider $\cD^{c,0}_j$. For any $k\in\NN_0$ with $k<\frac{\log(Q/Q_-)}{\log{R}},$
we define $$
\cD^{c,0}_{k,j}:=\left\{
\Delta\in \cD^{c,0}_j: R^{-(k+1)}Q<q(\Delta)\leq R^{-k}Q
\right\}.
$$
It follows that \begin{equation}\label{E:F0par}
    \cD^{c,0}_j=\bigsqcup_{k<\frac{\log(Q/Q_-)}{\log{R}}}\cD^{c,0}_{k,j}.
\end{equation}
Note that for any $\Delta\in \cD_{k,j}^{c,0}$, $$
4R^2(c(\Delta)+\rho'R^{j(\Delta)})\leq 4R^2(\eta R^{-k/2}+\rho'R^j)\overset{\eqref{E:jcondition}}{<}1.
$$
So applying Lemma \ref{L:simplex0} to each $\cD^{c,0}_{k,j}$ gives two affine line neighborhoods of radius $$
r_{k,j}=\frac{3\eta R^{3-k/2}}{2R^j},
$$
whose union covers all dangerous sets in $\cD^{c,0}_{k,j}$. For each nonempty half-family, Lemma \ref{L:simplex0} supplies an attached line $L_*:A_*x+B_*y+C_*=0$. With its label $q_*$ and $\xi_*$, \eqref{E:simplex-label} gives $|A_*\cdot B_*|\leq q_*$ and
$q_*\xi_*<R^{j+1}$.  Therefore
\[
r_{k,j}\xi_*|A_*\cdot B_*|\leq r_{k,j}\xi_*q_*
<\frac32\eta R^{4-k/2}
\leq\frac32\eta R^4.
\]
 In view of \eqref{E:F0par}, we conclude that the total potential of affine line neighborhoods to cover all dangerous sets in $\cD^{c,0}_j$ is \begin{equation*}
\sum_{k<\frac{\log(Q/Q_-)}{\log{R}}}2r_{k,j}^{\gamma}<\frac{2^{1-\gamma}(3\eta R^{3-j})^{\gamma}}{1-R^{-\gamma/2}}\overset{\eqref{E:Rcondition}}{<}\frac{1}{3}(\beta^2\theta R^{-j})^{\gamma}\overset{\eqref{E:defF}}{<}\frac{1}{3}(\beta\rho')^{\gamma}.
\end{equation*}

    Now we consider the $\cD^{c,1}_j$ case. The $\cD^{c,2}_j$ case is analogous. Put
    \[
    c_{*,j}:=\eta\min\left\{1,
    \sqrt{\frac{R^{(j+1)/2}}{Q}}\right\}.
    \]
    As before, every $\Delta\in\cD^{c,1}_j$ satisfies
    $c(\Delta)\leq c_{*,j}$. Thus $$
    2R(c_{*,j}+\rho'R^{j(\Delta)})\leq 2R(\eta+\rho' R^j)\overset{\eqref{E:jcondition}}{<}1.
    $$
    So applying Lemma \ref{L:simplex1} to each $\cD^{c,1}_j$ gives an vertical line neighborhood of radius 
    $$
    r_j=\frac{\eta}{R^j}\cdot\sqrt{\frac{R^{\frac{j+1}{2}}}{Q}},
$$
which covers all dangerous sets in $\cD^{c,1}_j$. We conclude that its total potential is \begin{equation*}
    r_j^{\gamma}\leq (\eta R^{1/4-j})^{\gamma}\overset{\eqref{E:Rcondition}}{<}\frac{1}{3}(\beta^2\theta R^{-j})^{\gamma}\overset{\eqref{E:defF}}{<}\frac{1}{3}(\beta\rho')^{\gamma}.
\end{equation*}
    
    Finally, putting all these affine line neighborhoods for $\cD^c_j=\cD_j^{c,0}\cup\cD_j^{c,1}\cup\cD_j^{c,2}$ together gives the desired conclusion.
\end{proof}

\section{The Winning Strategy and Proof of Theorem \ref{T:main}} \label{S:win} In this section, we play the hyperplane potential game on $\RR^2$ (see Section \ref{SS:HPW}) with the target set $E$, and show that Alice has a winning strategy. This will complete the proof of Theorem \ref{T:main}.

\subsection{The choice of $\eta$ and Alice's initial move} \label{SS:initial}
Let $\beta\in (0,1)$ and $\gamma>0$ be parameters, and let $\varphi(\cdot)$ denote the total $\gamma$-potential function as in \eqref{E:potential}. We choose an integer $R>\max\{4,\beta^{-1}\}$ sufficiently large so that \eqref{E:Rcondition} holds. In the following, we play the $(\beta,\gamma)$-hyperplane potential game on $\RR^2$ with the target set $E$. 

Let Bob's initial ball be $B_0=B(z_0,\rho_0)$. We choose $\eta>0$ sufficiently small so that \eqref{E:eta} holds and \begin{equation}\label{E:eta1}
    \sum_{q=1}^R\left(2\left(\frac{\eta}{\sqrt{R}}+\rho_0q\right)+1\right)\cdot\left(\frac{\eta}{q\sqrt{R}}\right)^{\gamma}\leq \frac{1}{2}(\beta\rho_0)^{\gamma}.
\end{equation}
Write $Q_0=R$. Our goal is to design Alice's initial collection $\cL_0$ so that \begin{equation}\label{E:0danger}
    \left\{
    (x,y)\in B_0: Q_0\cdot \min_{1\leq q\leq Q_0}\la qx\ra\la qy\ra<\eta^2
    \right\}\subseteq \bigcup_{L^{(r)}\in \cL_0}L^{(r)},
\end{equation}
and
\begin{equation}\label{E:0potential}
    \varphi(\cL_0)\leq (\beta\rho_0)^{\gamma}.
\end{equation}

Here is Alice's initial move. For each $P=(\frac{p}{q},\frac{r}{q})\in \cP$ with $q\leq Q_0$, we write \begin{equation*}
    L^1_P:=\left\{(x,y)\in\RR^2:x=\frac{p}{q}\right\};\qquad
    L^2_P:=\left\{(x,y)\in\RR^2:y=\frac{r}{q}\right\}.
\end{equation*}
Then we define Alice's initial collection $\cL_0$ to be \begin{equation}
    \bigcup_{i=1,2}\left\{
    (L^i_P)^{(\eta/(q\sqrt{Q_0}))}: P=\left(\frac{p}{q},\frac{r}{q}\right)\in\cP,q\leq Q_0, (L^i_P)^{(\eta/(q\sqrt{Q_0}))}\cap B_0\neq \varnothing
    \right\}.
\end{equation}
This collection is finite: $q\leq Q_0$, and the intersection
condition bounds $p$ for the distinct vertical lines and $r$ for the distinct horizontal lines.
We verify \eqref{E:0danger}: let $(x,y)\in B_0$ satisfy $$
Q_0\la qx\ra\la qy\ra<\eta^2\quad\text{ for some }1\leq q\leq Q_0.
$$
It follows that either $\sqrt{Q_0}\la qx\ra<\eta$, or $\sqrt{Q_0}\la qy\ra<\eta$. So we may choose $(p,r)\in\ZZ^2$ with $\gcd(p,r,q)=1$ such that for $P=(\frac{p}{q},\frac{r}{q})$, $$
(x,y)\in (L^1_P)^{(\eta/(q\sqrt{Q_0}))}\cup (L^2_P)^{(\eta/(q\sqrt{Q_0}))}.
$$
Since $(x,y)\in B_0$, the condition $(x,y)\in (L^i_P)^{(\eta/(q\sqrt{Q_0}))}$ automatically implies that $(L^i_P)^{(\eta/(q\sqrt{Q_0}))}\cap B_0\neq\varnothing$. This verifies \eqref{E:0danger}.

Now we calculate the total potential of the collection $\cL_0$. Note that every vertical line neighborhood $(L^1_P)^{(\eta/(q\sqrt{Q_0}))}$ in $\cL_0$ satisfies $$
1\leq q\leq Q_0\quad \text{and}\quad \left|\frac{p}{q}-z_{0,1}\right|\leq \frac{\eta}{q\sqrt{Q_0}}+\rho_0.
$$
So the total potential of vertical line neighborhoods in $\cL_0$ is not exceeding $$
\sum_{q=1}^{Q_0}\left(2\left(\frac{\eta}{\sqrt{Q_0}}+\rho_0q\right)+1\right)\cdot\left(\frac{\eta}{q\sqrt{Q_0}}\right)^{\gamma}
\overset{\eqref{E:eta1}}{\leq} \frac{1}{2}(\beta\rho_0)^{\gamma}.
$$
The same estimate holds for the total potential of horizontal line neighborhoods in $\cL_0$. Combining these two estimates gives  \eqref{E:0potential}.

\subsection{Definition of stages and Alice's general moves}\label{SS:general} In this section, we divide all turns of the game into stages, and design Alice's strategy of choosing affine line neighborhoods at each stage. This amounts to giving a partition of $\NN_0$: $$
\NN_0=\bigsqcup_{m\geq 0}\{N_{m-1}+1,\cdots,N_m\},\quad \text{ where }N_{-1}=-1<N_0<N_1<\cdots \text{ are integers}.
$$
The $m$-th stage refers to the turns in $\{N_{m-1}+1,\cdots,N_m\}$.

We set the $0$-th stage to be $\{0\}$, that is, $N_0=0$. At the $0$-th stage, write $Q_0=R$ and let Alice choose $\cL_0$ as her collection of affine line neighborhoods. Now let $m\geq 1$ and suppose that the $(m-1)$-th stage has been completed with the parameters $N_{m-1}$ and $Q_{m-1}$. 

At the $m$-th stage, our goal is to set two parameters $N_m>N_{m-1}$ and $Q_m>Q_{m-1}$, and to design Alice's choice of collections 
\begin{equation}\label{E:mcollection}
\bigcup_{N_{m-1}<n\leq N_m}\cL_n,
\end{equation}
so that the union of all affine line neighborhoods in these collections covers
\begin{equation}\label{E:mdanger}
 B_{N_m}\cap \bigcup_{\substack{P\in\cP,\\ Q_{m-1}<q\leq Q_m}}
 \Bigl(\bigcup_{s\in [0,1]}\Delta_{c(P,Q_m,\eta)}(P;s)\cup \Delta^1_{c(P,Q_m,\eta)}(P)\cup \Delta^2_{c(P,Q_m,\eta)}(P)\Bigr),
\end{equation}
and \begin{equation}\label{E:npotential}
    \varphi(\cL_n)\leq (\beta\rho_n)^{\gamma},\quad n\in \{N_{m-1}+1,\cdots,N_m\}.
\end{equation}
Here and hereafter, $\rho_n$ denotes the radius of Bob's ball at the $n$-th turn.

Here are Alice's general moves at the $m$-th stage. Given Bob's ball $B_{N_{m-1}+1}=B(z_{N_{m-1}+1},\rho_{N_{m-1}+1})$, we choose $Q_m>Q_{m-1}$ sufficiently large so that $$
Q_m\geq \frac{\theta^2R^4}{\rho_{N_{m-1}+1}^2Q_{m-1}}
,\quad \text{where}\quad \theta=\eta R^4.
$$

In the $(N_{m-1}+1)$-th turn, Alice deals with the dangerous family $$\cD_{N_{m-1}+1}:=\cD^s(Q_{m-1},Q_m;B_{N_{m-1}+1}).$$ Proposition \ref{P:potential1} gives a finite family $$\cL_{N_{m-1}+1}:=\cL^s(Q_{m-1},Q_m;B_{N_{m-1}+1})$$ of affine line neighborhoods so that $$
\bigcup_{\Delta\in \cD_{N_{m-1}+1}}\Delta\subseteq \bigcup_{L^{(r)}\in \cL_{N_{m-1}+1}}L^{(r)},
\quad 
\text{and} \quad
\varphi(\cL_{N_{m-1}+1})<(\beta\rho_{N_{m-1}+1})^{\gamma}.
$$

In other turns of the $m$-th stage, Alice deals with the remaining dangerous family
\begin{equation}\label{E:defRm}
\cR_m:=\left\{
\Delta\in\widetilde{\cD}\;\middle|\;
\begin{aligned}
&Q_{m-1}<q(\Delta)\leq Q_m,\quad c(\Delta)=c(P,Q_m,\eta),\\
&\rho_{N_{m-1}+1}\cdot R^{j(\Delta)}\geq \theta R,\quad \Delta\cap B_{N_m}\neq \varnothing
\end{aligned}
\right\}.
\end{equation}
Here $N_m>N_{m-1}$ is to be determined below. Write \begin{equation}\label{E:triggerj}
J_m:=\{j\in\NN_0:R^j\leq Q_m^2\text{ and }
\rho_{N_{m-1}+1}\cdot R^j\geq \theta R\}.
\end{equation}
For any $j\in J_m$, we define its trigger time \begin{equation*}
    \tau(j):=\inf\{n>N_{m-1}+1:\rho_n\leq \theta R^{-j}\}.
\end{equation*}
Here the convention $\inf\varnothing=+\infty$ is adopted. It follows that when $\tau(j)<+\infty$, \begin{equation}\label{E:trigger}
\rho_{\tau(j)}\leq \theta R^{-j}<\rho_{\tau(j)-1}\leq \beta^{-1}\rho_{\tau(j)}.
\end{equation}
We make two more simple observations: \begin{enumerate}
    \item[($\tau1$)] If $\rho_n\to 0$ as $n\to +\infty$, then for any $j\in J_m$, $\tau(j)<+\infty$;
    \item[($\tau2$)] If $j\neq j'\in J_m$ and $\tau(j),\tau(j')<+\infty$, then $\tau(j)\neq \tau(j')$.
\end{enumerate}

We set $$
N_m:=\max\{N_{m-1}+1,\tau(j) \text{ where }j\in J_m\}.
$$
Then $N_m<+\infty$ as long as all $\tau(j)<+\infty$.
It is clear that for Bob's balls $B_n$ where $N_{m-1}+1<n\leq N_m$, one has \begin{equation}
    \cR_m\subseteq \bigcup_{j\in J_m}\cD^c_j(Q_{m-1},Q_m;B_{\tau(j)}).
\end{equation}

If $\tau(j)<+\infty$, then in the $\tau(j)$-th turn, Alice deals with the dangerous family $$
\cD_{\tau(j)}:=\cD^c_j(Q_{m-1},Q_m;B_{\tau(j)}).
$$
Proposition \ref{P:potential2} gives a finite family $$
\cL_{\tau(j)}:=\cL^c_j(Q_{m-1},Q_m;B_{\tau(j)})$$
of affine line neighborhoods so that $$
\bigcup_{\Delta\in \cD_{\tau(j)}}\Delta\subseteq \bigcup_{L^{(r)}\in \cL_{\tau(j)}}L^{(r)},
\quad 
\text{and} \quad
\varphi(\cL_{\tau(j)})<(\beta\rho_{\tau(j)})^{\gamma}.
$$

Finally, in the $n$-th turn where 
$$n\in \{N_{m-1}+2,\cdots,N_m\}\setminus\{\tau(j):j\in J_m\},$$ Alice chooses $\cL_n=\varnothing$. Therefore, we conclude that as soon as $\rho_n\to 0$ as $n\to +\infty$, Alice's collections \eqref{E:mcollection}
satisfy the requirements \eqref{E:mdanger} and \eqref{E:npotential}. 

\subsection{Proof of Theorem \ref{T:main}} 
By Lemma \ref{L:HAWHPW}, it suffices to show that the set $E$ is hyperplane potential winning. Let the $(\beta,\gamma)$-hyperplane potential game go as in Sections \ref{SS:initial} and \ref{SS:general}. If Alice doesn't win by default, then one has $\rho_n\to 0$ as $n\to +\infty$, and $$
\bigcap_{n\geq 0}B_n\cap \bigcup_{n\geq 0}\bigcup_{L^{(r)}\in \cL_n}L^{(r)}=\varnothing.
$$
It follows from \eqref{E:0danger} and \eqref{E:mdanger} that the outcome of $\bigcap_{n\geq 0}B_n$ avoids \eqref{E:Eeta0} and \eqref{E:Eeta1}. By Corollary \ref{C:Eeta}, we see that this outcome must lie in $E$. So Alice wins.

\section{Proof of Theorem \ref{T:1+epsilon} and Related Results}\label{S:further}
In this section, we deduce Theorem \ref{T:1+epsilon} from our proof of Theorem \ref{T:main}, using a common trick in the study of winning sets. 
We start with a simple observation that is claimed in Section \ref{SS:refined}. 

\begin{lemma}\label{L:LCULC}
    Let $\bQ=(Q_m)_{m\geq 0}$ be a strictly increasing sequence of positive real numbers tending to infinity. Then $E'\subseteq E_{\bQ}\subseteq E$. Moreover, when $\bQ$ has bounded ratios, one has $E'=E_{\bQ}$.
\end{lemma}
\begin{proof}
First, it is straightforward to see that $E_{\bQ}\subseteq E$.
    We next prove that $E'\subseteq E_{\bQ}$. Let $(x,y)\in E'$, that is, there exists some $c>0$ and $N_0>0$ such that for any $q>N_0$, $$
    q\la qx\ra\la qy\ra>c.
    $$
    In particular, $x,y\notin\QQ$. We write $c_0:=\min_{1\leq q\leq N_0}q\la qx\ra\la qy\ra>0$. Then for any $m\geq 0$ with $Q_m>N_0$, we have $$
    Q_m\cdot \min_{1\leq q\leq Q_m}\la qx\ra\la qy\ra
    \geq \min_{1\leq q\leq Q_m}q\la qx\ra\la qy\ra
    \geq\min\{c_0,c\}>0.
    $$
    This means that $\liminf_{m\to +\infty}Q_m\cdot \min_{1\leq q\leq Q_m}\la qx\ra\la qy\ra>0$, that is, $(x,y)\in E_{\bQ}$. 

    Now we suppose that $\bQ$ has bounded ratios, that is, there exists some $R>0$ such that $\sup_{m\geq 0}\frac{Q_{m+1}}{Q_m}\leq R$. For any $(x,y)\in E_{\bQ}$, there exists some $\eta>0$ and $m_0\geq 1$ such that $$
Q_m\cdot \min_{1\leq q\leq Q_m}\la qx\ra\la qy\ra\geq \eta,\quad \forall m\geq m_0.
$$
Then for each $q>Q_{m_0-1}$, we define $$
m(q):=\min\{m\geq m_0:q\leq Q_m\}.
$$
It follows that $R^{-1}Q_{m(q)}\leq Q_{m(q)-1}<q\leq Q_{m(q)}$. Therefore, we conclude that $$
    q\la qx\ra\la qy\ra>R^{-1}Q_{m(q)}\cdot \min_{1\leq n\leq Q_{m(q)}}\la nx\ra\la ny\ra\geq R^{-1}\eta.
$$
This means that $\liminf_{q\to +\infty}q\la qx\ra\la qy\ra>0$, that is, $(x,y)\in E'$.
\end{proof}

Next, we focus on Problem \ref{P:growth}, which aims to study the Hausdorff dimension of $E_{\bQ}$ when a growth rate of the sequence $\bQ$ is given. By carefully examining the proof of Theorem \ref{T:main}, we see that there exists a specific sequence $\bQ=(Q_m)_{m\geq 0}$ satisfying \begin{equation}
    Q_{m+1}>R^2Q_m^3,\quad m\geq 0
\end{equation}
so that the resulting set $E_{\bQ}$ is $(\beta,\gamma)$-hyperplane potential winning on $\RR^2$. Here the constant $R>0$ explicitly depends on the parameters $\beta\in (0,1)$ and $\gamma>0$. Now we formally generalize this observation.




\begin{lemma}\label{L:cubic}
    Let $\bQ=(Q_m)_{m\geq 0}$ be any strictly increasing sequence of positive real numbers with \begin{equation}\label{E:cubic}
        \lim_{m\to +\infty}\frac{Q_{m+1}}{Q_m^3}=+\infty.
    \end{equation}
    Then the set $E_{\bQ}$ is hyperplane absolute winning on $\RR^2$.
    
    More precisely, given $\beta\in(0,1)$, $\gamma>0$, an initial
    ball $B_0$, and $\varepsilon>0$, Alice has a
    $(\beta,\gamma)$-hyperplane-potential strategy with the following
    properties.  At turn $n$ she deletes a finite family
    \[
       \cL_n=\{L_{n,i}^{(r_{n,i})}:i\in I_n\}.
    \]
    Each $L_{n,i}$ has a primitive integral equation
    \[
       A_{n,i}x+B_{n,i}y+C_{n,i}=0,\qquad
       \xi_{n,i}:=\max\{|A_{n,i}|,|B_{n,i}|\},
    \]
    satisfying
    \begin{equation}\label{E:structured-control}
       r_{n,i}\xi_{n,i}|A_{n,i}B_{n,i}|<\varepsilon.
    \end{equation}
    If the radii of Bob's balls tend to zero and the outcome avoids every
    neighborhood $L_{n,i}^{(r_{n,i})}$, then the outcome belongs to
    $E_{\bQ}$.
\end{lemma}
\begin{proof}
    Fix $\beta,\gamma,B_0$, and $\varepsilon$ as in the statement.
    Choose an integer $R>\max\{4,\beta^{-1}\}$ sufficiently large so that \eqref{E:Rcondition} holds. By \eqref{E:cubic}, for all sufficiently large $m\geq 0$,  one has \begin{equation}\label{E:cubic2}
    Q_{m+1}\geq R^4\beta^{-4}Q_m^3.
    \end{equation}
    Since the definition of $E_{\bQ}$ involves $\liminf_{m\to +\infty}$, we may assume that \eqref{E:cubic2} holds for all $m\geq 0$. Moreover, we may also assume that $Q_0=R$.

    Given Bob's initial ball $B_0=B(z_0,\rho_0)$, we choose $\eta>0$ sufficiently small so that \eqref{E:eta} and \eqref{E:eta1} hold, as well as
    \begin{equation}\label{E:eta2}
        \frac32\eta R^4<\varepsilon,\quad\text{and}\quad\frac{\theta^2R^4}{\beta^2\rho_0^2Q_0}\leq Q_1,\quad \text{where }\theta=\eta R^4.
    \end{equation}
    Then Alice's initial choice $\cL_0$ goes as in Section \ref{SS:initial}, and write $N_0=0$.

    Let $m\geq 1$ and let the $m$-th stage begin at the turn $N_{m-1}+1$. Given Bob's ball $B_{N_{m-1}+1}=B(z_{N_{m-1}+1},\rho_{N_{m-1}+1})$, we consider the key inequality: \begin{equation*}
        \frac{\theta^2R^4}{\rho_{N_{m-1}+1}^2Q_{m-1}}\leq Q_m\tag{$I_m$}.
    \end{equation*}
    We shall show that: \begin{equation}\label{E:induction}
        (I_m)\Longrightarrow \text{completion of the $m$-th stage}\Longrightarrow (I_{m+1}).
    \end{equation}
    Note that the inequality \eqref{E:eta2} gives $(I_1)$. So by induction, all our stages will be completed. 
    
    For the first implication in \eqref{E:induction}, given $(I_m)$, we proceed as in Section \ref{SS:general} to choose $N_m>N_{m-1}$ and Alice's collections \eqref{E:mcollection} satisfying the requirements \eqref{E:mdanger} and \eqref{E:npotential}, as long as $\rho_n\to 0$ as $n\to +\infty$.
    
    For the second implication in \eqref{E:induction}, we note that in the process of the $m$-th stage, $$
    j_m:=\left\lfloor \frac{2\log{Q_m}}{\log{R}}\right\rfloor
    $$
    is the largest number in $J_m$. 
    So $N_m=\tau(j_m)$. In particular, we see that $$
    \rho_{N_m+1}\geq \beta\rho_{N_m}=\beta\rho_{\tau(j_m)}\overset{\eqref{E:trigger}}{>}\beta^2\theta R^{-j_m}\overset{\eqref{E:triggerj}}{\geq} \beta^2\theta Q_m^{-2}.
    $$
    It follows that $$
    \frac{\theta^2R^4}{\rho_{N_m+1}^2Q_m}<R^4\beta^{-4}Q_m^3\overset{\eqref{E:cubic2}}{\leq}Q_{m+1}.
    $$
    This verifies $(I_{m+1})$ as desired.

    It remains to identify every nondefault outcome.  If the radii
    tend to zero and the outcome avoids all the neighborhoods Alice
    deleted, then \eqref{E:0danger} and \eqref{E:mdanger}, stage by
    stage, show that it avoids both sets \eqref{E:Eeta0} and
    \eqref{E:Eeta1}.  The strengthened form of Corollary \ref{C:Eeta}
    therefore places it in $E_{\bQ}$.  Thus the displayed strategy has
    all the asserted properties and is winning.  Since $\beta$ and
    $\gamma$ were arbitrary, Lemma \ref{L:HAWHPW} yields that
    $E_{\bQ}$ is HAW. The proof is complete.
\end{proof}

Furthermore, using the fact that any finite intersection of HAW sets is again HAW, we can prove Theorem \ref{T:1+epsilon} which substantially strengthens Lemma \ref{L:cubic}.

\begin{proof}[Proof of Theorem \ref{T:1+epsilon}]
We first slightly generalize Lemma \ref{L:cubic} as follows: 
\begin{equation}\label{E:gencubic}
    \lim_{m\to+\infty}\frac{Q_{m+k}}{Q_m^3}=+\infty\text{ for some }k\geq 1\Longrightarrow E_{\bQ}\text{ is HAW on }\RR^2.
\end{equation}
In fact, for $k\geq 1$, consider the $k$ subsequences of $\bQ=(Q_m)_{m\geq 0}$: $$
\bQ^{(r)}:=(Q_{r+kj})_{j\geq 0},\quad r\in \{0,1,\cdots,k-1\}.
$$
By the condition, each subsequence $\bQ^{(r)}$ satisfies $$
\lim_{j\to +\infty}\frac{Q_{r+k(j+1)}}{Q_{r+kj}^3}=+\infty,
$$
so it follows from Lemma \ref{L:cubic} that each $E_{\bQ^{(r)}}$ is HAW on $\RR^2$. Therefore, the finite intersection $$
E_{\bQ}=\bigcap_{r=0}^{k-1}E_{\bQ^{(r)}}
$$
is also HAW on $\RR^2$. This proves \eqref{E:gencubic}.

Now suppose that the sequence $\bQ=(Q_m)_{m\geq 0}$ satisfies the condition \eqref{E:1+epsilon}, that is, there exists some $c>0$ and $\tau>1$ such that $Q_{m+1}\geq c Q_m^{\tau}$ for all $m\geq 0$. We choose some $k\geq 1$ such that $\tau^k>3$. It follows that $$
Q_{m+k}\geq c^{1+\tau+\cdots+\tau^{k-1}}\cdot Q_m^{\tau^k}= (c^{1+\tau+\cdots+\tau^{k-1}}Q_m^{\tau^k-3})\cdot Q_m^3.
$$
Since $Q_m\to +\infty$ as $m\to +\infty$, we see that $\frac{Q_{m+k}}{Q_m^3}\to +\infty$ as $m\to +\infty$. So we conclude from \eqref{E:gencubic} that $E_{\bQ}$ is HAW on $\RR^2$. This completes the proof.
\end{proof}


\section{Proofs of Theorems \ref{T:main-curves} and \ref{T:lines}}\label{S:curves}

 By the definition of HAW on a manifold in Section \ref{hawmfld}, it is enough to work on an arbitrary coordinate arc of $\mathcal C$. Let
$\psi:V\to\mathcal C$, where $V\subset\RR$ is an open interval, be a regular $C^2$ embedding of $V$ into $\mathcal C$.  Choose a
$C^2$ diffeomorphism $\sigma:\RR\to V$ and put
\[
   \phi=(X,Y):=\psi\circ\sigma.
\]
The curvature assumption means that the set
\begin{equation}\label{E:curve-determinant}
  \{ t \in \RR: \Delta_\phi(t) \neq 0 \} \qquad \text{where} \  \Delta_\phi(t):=X'(t)Y''(t)-Y'(t)X''(t)
\end{equation}
is absolute winning, equivalently, HAW.  The winning property is independent of the regular parametrization: if $\sigma$ is a regular $C^2$ change of parameter, then
\[
 \Delta_{\phi\circ\sigma}
   =(\sigma')^3(\Delta_\phi\circ\sigma).
\]

For $(A,B)\in \ZZ^2\setminus\{(0,0)\}$, put $\xi(A,B)=\max\{|A|,|B|\}$.  For $c>0$, define
\begin{equation}\label{E:Dc}
 D_c=\left\{t\in\RR:
 \begin{gathered}
 \xi(A,B)^2|AX(t)+BY(t)+C|\geq c\\
 \text{for all }(A,B,C)\in\ZZ^3\text{ with } (A,B)\ne(0,0)
 \end{gathered}\right\}.
\end{equation}
It is enough to use primitive triples in the above definition of $D_c$. Geometrically, the set $D_c$ consists of the parameters whose $\phi$-images stay outside the open tube about every primitive integral line $Ax+By+C=0$ of radius
\begin{equation}\label{E:Dc-radius}
 \frac{c}{\xi(A,B)^2\sqrt{A^2+B^2}}\asymp\frac{c}{\xi(A,B)^3}.
\end{equation}

Fix a sequence $\bQ$ satisfying \eqref{E:cubic}, and let
\begin{equation}\label{E:ScQ}
 S_{\bQ,c}=\phi^{-1}(E_{\bQ})\cup(\RR\setminus D_c).
\end{equation}

Our first goal is to transfer the strategy from Lemma \ref{L:cubic} on curves and prove that every such set is winning.

\begin{lemma}\label{P:conditional-curves}
For every $c>0$, the set $S_{\bQ,c}$ is HAW on $\RR$.
\end{lemma}

The lemma is based on the following two lemmata.

\begin{lemma}\label{L:tube-pullback}
Let $J\subset\RR$ be an open interval such that
\begin{equation}\label{E:curve-derivative-bounds}
  m\leq |X'(t)|,|Y'(t)|\leq M\qquad \text{for all} \ t\in J
\end{equation}
 for some constants $0<m\leq M$, and put
\begin{equation}\label{E:curve-pullback-constants}
 \varepsilon=\frac{cm}{2\sqrt2\,M}.
\end{equation}
Let $L:Ax+By+C=0$ be a line with the primitive integral triple $(A,B,C)$. 
Set $\xi=\max\{|A|,|B|\}$, $F(t)=AX(t)+BY(t)+C$ for $t\in J$, and choose $r>0$ with
\begin{equation}\label{E:curve-tube-control}
   r\xi|A\cdot B|<\varepsilon.
\end{equation}
Then for any closed interval $\mathcal B\subset J$,
\begin{enumerate}[label=\rm (\arabic*)]
\item\label{N:curve-pullback1} If $|F'(u)|\geq \xi m/2$ for every $u\in\mathcal B$, the set
      $\mathcal B\cap\phi^{-1}(L^{(r)})$ is contained in an
      interval of radius at most $2\sqrt{2}r/m$;
\item\label{N:curve-pullback2} If the derivative condition in part (1) fails, then
      $J\cap \phi^{-1}(L^{(r)})\subset\RR\setminus D_c$.
\end{enumerate}
Every vertical or horizontal line satisfies the derivative condition in
part~\ref{N:curve-pullback1}.
\end{lemma}

\begin{proof}
Suppose first that the derivative condition in part~\ref{N:curve-pullback1} holds.  If
$t,t'\in\mathcal B\cap\phi^{-1}(L^{(r)})$, then
\begin{equation}\label{E:t_in_Lr}
|F(t)|,|F(t')|\leq\sqrt{A^2+B^2}\cdot r\leq\sqrt2\xi r.
\end{equation}
The continuous function $F'$ has constant sign on $\mathcal B$, so the
mean-value theorem gives
\[
 \xi m|t-t'|/2\leq |F'(u)| |t - t'| =  |F(t)-F(t')|\leq2\sqrt2 \xi r
\]
for some $u$ between $t$ and $t'$. Thus the set $\mathcal B\cap\phi^{-1}(L^{(r)})$ has diameter at most $4\sqrt2r/m$, and hence is contained in an interval of radius at most $2\sqrt2r/m$.

Now suppose that the derivative condition fails, that is, there is $u\in\mathcal B$ such that
\begin{equation}\label{E:curve-near-tangent}
 |AX'(u)+BY'(u)|<\xi m/2.
\end{equation}
If $\xi=|A|$, then
\[
 |B|\,|Y'(u)|
 \geq |A|\,|X'(u)|-|AX'(u)+BY'(u)|
 >\xi m-\xi m/2=\xi m/2.
\]
Since $|Y'(u)|\leq M$, this implies
\begin{equation}\label{E:AB>}
|A\cdot B|> \xi^2m/(2M).
\end{equation}
The same conclusion follows symmetrically if
$\xi=|B|$.  We conclude that for every $t\in J$ with $\phi(t)\in L^{(r)}$, one has
\begin{equation}
 \xi^2|F(t)|
 \overset{\eqref{E:t_in_Lr}}{\leq}\sqrt2\xi^3r
 =\sqrt2 \left(r \xi|A\cdot B|\right)\frac{\xi^2}{|A\cdot B|}
 \overset{\eqref{E:curve-tube-control}, \eqref{E:AB>}}{<}\sqrt2\varepsilon\frac{2M}{m}\overset{\eqref{E:curve-pullback-constants}}{=}c. 
\end{equation}
This implies that $t\notin D_c$. The proof of part~\ref{N:curve-pullback2} is complete.  

Finally, if $A=0$ or $B=0$, then $|F'(u)|=\xi|Y'(u)|$ or $\xi|X'(u)|$, which is at least $\xi m$ due to \eqref{E:curve-derivative-bounds}.  Thus a vertical or horizontal line always falls under part~\ref{N:curve-pullback1}.
\end{proof}

Put
\begin{equation}\label{E:curve-exceptional}
 Z=\{t\in\RR:X'(t)Y'(t)=0\}.
\end{equation}

\begin{lemma}\label{L:local-conditional}
For every $c>0$ and every bounded open interval $J$ with
$\overline J\subset\RR\setminus Z$, the set $S_{\bQ,c}$ is HAW on $J$.
\end{lemma}

\begin{proof}
Choose $m,M$ so that \eqref{E:curve-derivative-bounds} holds on
$\overline J$, 
and fix arbitrary parameters $\beta\in(0,1)$ and $\gamma>0$ for the one-dimensional potential game. We choose $0<\widetilde\beta\leq\beta$ such that
\begin{equation}\label{E:curve-beta-tilde}
\frac{4M}{m}\cdot\widetilde\beta\leq\beta.
\end{equation}
Let Bob's $n$-th interval be $B_n=B(t_n,\rho_n)\subset J$, and define the
ambient shadow ball
\begin{equation}\label{E:curve-shadow-ball}
 \widehat B_n=B\bigl(\phi(t_n),\sqrt2M\rho_n\bigr)\subset\RR^2.
\end{equation}
These balls $\{\widehat B_n:n\geq 0\}$ are nested.  Indeed, the fact (from $B_{n+1}\subseteq B_n$)
\[
 |t_{n+1}-t_n|+\rho_{n+1}\leq\rho_n
\]
together with $\phi$ being $\sqrt{2}M$-Lipschitz (from the derivative bounds) implies
\[
 |\phi(t_{n+1})-\phi(t_n)|+\sqrt2M\rho_{n+1}
 \leq\sqrt2M\rho_n.
\]
Their radius ratio is at least $\beta\geq\widetilde\beta$, so they are
legal Bob moves in the ambient $(\widetilde\beta,\gamma)$-potential game. We will refer to it simply as \emph{the ambient game}.

We apply the structured strategy of Lemma \ref{L:cubic} to the balls $\widehat B_n$ with $\epsilon$ as in \eqref{E:curve-pullback-constants}.  At turn $n$ it
produces a finite family
\[
 \{L_{n,i}^{(r_{n,i})}:i\in I_n\},\qquad
 L_{n,i}:A_{n,i}x+B_{n,i}y+C_{n,i}=0,
\]
whose coefficient triples are primitive and satisfy \eqref{E:structured-control}.
For every $i\in I_n$ define
\begin{equation}\label{E:curve-Fni}
 F_{n,i}(t)=A_{n,i}X(t)+B_{n,i}Y(t)+C_{n,i}.
\end{equation}
Let $I_n^{\mathrm{tr}}$ consist of those $i\in I_n$ for which the set
\begin{equation}\label{E:curve-tube-parameter-set}
 \left\{t\in B_n: \phi(t) \in L_{n,i}^{(r_{n,i})} \right\} = \left\{t\in B_n:
 |F_{n,i}(t)|\leq
 \sqrt{A_{n,i}^2+B_{n,i}^2}\cdot r_{n,i}\right\}
\end{equation}
is nonempty and
\[
 |F_{n,i}'(u)|\geq
 \max\{|A_{n,i}|,|B_{n,i}|\}\cdot m/2\qquad \text{for all} \ u\in B_n.
\]
For $i\in I_n^{\mathrm{tr}}$, Lemma \ref{L:tube-pullback} \ref{N:curve-pullback1}, applied with $\mathcal{B} = B_n$, supplies an interval $\mathcal A_{n,i}$ containing
\eqref{E:curve-tube-parameter-set}, of radius
$\varrho_{n,i}\leq 2\sqrt{2}r_{n,i}/m$. Let Alice delete precisely these intervals. This move is legal because
\begin{equation}\label{E:curve-potential-budget}
 \sum_{i\in I_n^{\mathrm{tr}}}\varrho_{n,i}^{\gamma} \leq \left(\frac{2\sqrt{2}}{m}\right)^\gamma\cdot \sum_{i\in I_n}r_{n,i}^{\gamma} \leq\left(\frac{4M\widetilde{\beta}\rho_n}{m}\right)^\gamma \overset{\eqref{E:curve-beta-tilde}}{\leq}(\beta\rho_n)^\gamma,
\end{equation}
where the middle inequality holds because of Alice's move in the ambient game. 

It remains to show that Alice's strategy described above is winning. Suppose that the radii shrink to zero and that the outcome $t_\infty$ lies in none of Alice's deleted intervals.  If $t_\infty\notin D_c$, then $t_\infty\in S_{\bQ,c}$. Now we assume that $t_\infty\in D_c$.  We claim that 
\begin{equation}\label{E:outcome_of_ambient}
\phi(t_\infty)\notin L_{n,i}^{(r_{n,i})} \qquad \text{for all} \ n \ge 0 \text{ and all } i \in I_n.
\end{equation}
Suppose, to the contrary, that
$\phi(t_\infty)\in L_{n,i}^{(r_{n,i})}$ for some $n\geq 0$ and $i\in I_n$. Then
$t_\infty\in B_n$ belongs to the set
\eqref{E:curve-tube-parameter-set}. Two options are possible:
\begin{itemize}
    \item If the derivative
condition defining $I_n^{\mathrm{tr}}$ holds, then by construction $t_\infty\in\mathcal A_{n,i}$, which contradicts the assumption that $t_\infty$ does not belong to Alice's deleted intervals.  \item If it fails, Lemma
\ref{L:tube-pullback} \ref{N:curve-pullback2} yields $t_\infty\notin D_c$, contradicting the assumption.  
\end{itemize}

These two contradictions prove the claim \eqref{E:outcome_of_ambient}. Moreover, we note that the radii $\sqrt2M\rho_n$ in the ambient game tend to zero, and 
that $\phi(t_\infty)\in\widehat B_n$ for every $n$. Thus $\phi(t_\infty)$ is the outcome of the ambient game. It follows from Lemma \ref{L:cubic} that $\phi(t_\infty)\in E_{\bQ}$ and hence $t_\infty\in S_{\bQ,c}$.

This proves the potential winning property for games on $J$. Finally, the proof will be complete once we use the following elementary observation, valid for both the hyperplane absolute and hyperplane potential games: a set \(S\subset\RR\) is winning on an open interval \(J\) if and only if
$ S\cup(\RR\setminus J)$
is winning on \(\RR\). 
\end{proof}



\begin{proof}[Proof of Lemma \ref{P:conditional-curves}]
Fix $c>0$. The set $\RR \setminus Z$, where $Z$ in \eqref{E:curve-exceptional}, is HAW and open. Indeed, the set $\{ t \in \RR: \Delta_\phi(t) \neq 0 \}$, where $\Delta_\phi$ is defined in \eqref{E:curve-determinant}, is HAW by hypothesis.  If $X'(t_0)=0$, regularity gives $Y'(t_0)\ne0$; if $\Delta_\phi(t_0)\ne0$, then
\[
 \Delta_\phi(t_0)=-Y'(t_0)X''(t_0)\ne0.
\]
Thus $X''(t_0)\ne0$, so this zero of $X'$ is isolated.  A discrete subset of $\RR$ is countable, so the zeros of $X'$ outside the zero set of $\Delta_\phi$ form a countable set. The same argument applies to $Y'$; so $ \{t:\Delta_\phi(t)\ne0\}\cap(\RR\setminus Z) \subseteq \RR\setminus Z$ is an intersection of a HAW set and a co-countable (also HAW) set, thus it is HAW.  

Choose a countable cover of $\RR\setminus Z$ by bounded open intervals $J_k$ with $\overline{J_k}\subset\RR\setminus Z$; that is, $\{ (J_k, \id_{J_k}) \}$ is an atlas on $\RR\setminus Z$.  
Lemma \ref{L:local-conditional} proves that $S_{\bQ,c}$ is HAW on each $J_k$. By definition (see Section \ref{hawmfld}), $S_{\bQ,c}$ is HAW on $\RR\setminus Z$, which by Lemma \ref{L:HAW-open} implies that it is HAW on $\RR$.
\end{proof}

Let
\begin{equation}\label{E:curve-D}
 D=\bigcup_{c>0}D_c.
\end{equation} 

\begin{lemma}\label{L:D_is_bad}
One has
    \begin{equation}\label{E:curve-transference}
D=\phi^{-1}\!\left(\Bad\!\left(\tfrac12,\tfrac12\right)\right).
\end{equation}
\end{lemma}
\begin{proof}
    Recall that $\Bad\left(s,t \right)$ is defined via \eqref{E:bad-definition}. 
    It can be shown using the transference principle (see, for example, \cite[\S1.3]{BPV11}) that
    $$
    \Bad \left(\frac12,\frac12 \right) = \left\{
        (x,y)\in\RR^2: \inf_{(A, B) \in \ZZ^2 \setminus (0,0)}\max\{|A|^2, |B|^2\}\la Ax+By \ra>0
        \right\}.
    $$
    Recalling that $\max\{|A|, |B|\}=\xi(A,B)$, we have
    $$
    \Bad \! \left(\frac12, \frac12 \right) =\bigcup\limits_{c > 0} \left\{
        (x,y)\in\RR^2: \begin{gathered}
 \xi(A,B)^2|Ax+By+C|\geq c\\
 \text{for all }(A,B,C)\in\ZZ^3,\ (A,B)\ne(0,0)
 \end{gathered}
        \right\}.
    $$
    The latter form and the definition of $D_c$ immediately yield \eqref{E:curve-transference}.
\end{proof}

\begin{lemma}\label{L:D_is_HAW}
    The set $D$ is HAW on $\RR$.
\end{lemma}  

\begin{proof} In the special case when $X(t) = t$ and $Y''(t) \neq 0$, due to Lemma \ref{L:D_is_bad}, \cite[Theorem 2.1]{ABV18} proves that $D=\phi^{-1}\!\left(\Bad\!\left(\tfrac12,\tfrac12\right)\right)$ is $\frac{1}{2}$-winning. Due to \cite[Corollary 1.12]{BHNS25}, $D$ is HAW. 

The general case follows via a change of coordinates. Let
\[
 Z_1=\{t\in\RR:\Delta_\phi(t)X'(t)Y'(t)=0\}.
\]
It is closed, and $\RR \setminus Z_1$ is HAW: see the argument in the proof of Lemma \ref{P:conditional-curves}. Cover $\RR\setminus Z_1$ by bounded
intervals $J$ whose closures avoid $Z_1$.  On such an interval, $X$ is
a $C^2$ diffeomorphism onto its image, and the curve is the graph of
$f=Y\circ X^{-1}$.  Direct differentiation gives
\begin{equation}\label{E:curve-second-derivative}
 f''(X(t))=\frac{\Delta_\phi(t)}{X'(t)^3}\ne0.
\end{equation}
By \cite[Theorem 2.1]{ABV18}, the set
\[
 A_J=\{u\in X(\overline J):(u,f(u))\in
 \Bad(\tfrac12,\tfrac12)\}
\]
is $1/2$-winning on the compact interval $X(\overline J)$. Arguing as before, we conclude that the set
\[
 \widehat A_J=(A_J\cap X(J))\cup(\RR\setminus X(J))
\]
is $1/2$-winning on $\RR$.  By
\cite[Corollary 1.12]{BHNS25}, it is HAW.  Restricting the game back to $X(J)$ and using the $C^1$-invariance in Lemma \ref{L:HAW-mnfd}, we conclude that $D$ is HAW on every such $J$, thus on $\RR \setminus Z_1$. Using Lemma \ref{L:HAW-open}, we conclude that it is HAW on $\RR$.
\end{proof}

By \eqref{E:ScQ} and \eqref{E:curve-D}, one has \begin{equation}\label{E:curve-W}
 D \cap \phi^{-1}(E_{\bQ}) = D\cap\bigcap_{n=1}^{\infty}S_{\bQ,1/n}.
\end{equation}
Note that Lemma \ref{P:conditional-curves} shows that every
$S_{\bQ,1/n}$ is HAW, and Lemma \ref{L:D_is_HAW} shows that $D$ is HAW. So by countable-intersection stability $D\cap \phi^{-1}(E_{\bQ})$ is also HAW. Since a superset of a HAW set is HAW, we conclude that
\begin{equation}\label{E:cubic-on-curve}
 \phi^{-1}(E_{\bQ})\quad\text{is HAW}\quad \text{whenever \eqref{E:cubic} holds.}
\end{equation}

\begin{proof}[Proof of Theorem \ref{T:main-curves}] 
To deduce Theorem \ref{T:main-curves} from \eqref{E:cubic-on-curve}, we repeat the argument used in the proof of Theorem \ref{T:1+epsilon}.

Suppose \eqref{E:1+epsilon} holds. There exist $c>0$ and $\tau>1$ such that $Q_{m+1}\geq cQ_m^\tau$.  Choose $k$ with $\tau^k>3$ and, for $0\leq r<k$, put
\[
 \bQ^{(r)}=(Q_{r+kj})_{j\geq0}.
\]
Iterating the above inequality yields that 
\[
 \frac{Q_{r+k(j+1)}}{Q_{r+kj}^3}\longrightarrow\infty.
\]
Thus \eqref{E:cubic-on-curve} applies to every $\bQ^{(r)}$.  Since
\[
 E_{\bQ}=\bigcap_{r=0}^{k-1}E_{\bQ^{(r)}},
\]
finite-intersection stability shows that $\phi^{-1}(E_{\bQ})$ is HAW.
Since $\phi=\psi\circ\sigma$,
\[
 \psi^{-1}(E_{\bQ})=\sigma\bigl(\phi^{-1}(E_{\bQ})\bigr),
\]
which is HAW on $V$ by Lemma \ref{L:HAW-mnfd}(iii).  This holds in every
coordinate arc, and therefore
$E_{\bQ}\cap\mathcal C$ is HAW (equivalently, absolute winning) on $\mathcal C$.

To prove the result for $E$, note that Lemma
\ref{L:LCULC} gives $E_{\bQ}\subset E$, so $E\cap\mathcal C$ is HAW as a superset of $E_{\bQ} \cap \mathcal C$.
This proves Theorem \ref{T:main-curves}.  The dimension assertion follows from Lemma \ref{L:HAW-mnfd}(i).
\end{proof}

\begin{proof}[Proof of Theorem \ref{T:lines}]
    \begin{enumerate}
        \item Let us parameterize $\mathcal{C}$ by $\phi(t) = (X(t), Y(t))$ where $X(t) = t, Y(t)= at + b$. First, assume that \eqref{E:cubic} holds. The set \eqref{E:curve-exceptional} is empty, and thus by Lemma \ref{L:local-conditional} the set $S_{\bQ, c}$ is HAW on every bounded interval of $\RR$, thus on the whole $\RR$, for every $c > 0$. By \cite[Theorem 1.2]{ABV18} and \eqref{E:curve-transference}, the set $D$ defined via \eqref{E:curve-D} is $1/2$-winning, equivalently, HAW on $\RR$. We just established that the conclusions of Lemma \ref{P:conditional-curves} and Lemma \ref{L:D_is_HAW} hold for $\mathcal{C}$. The remainder of the proof of Theorem \ref{T:main-curves} holds verbatim.
        \item Fix $x \in \RR$, and let $y = ax + b$. If $x \in \QQ$, then $\min\limits_{1\leq q\leq Q}\la qx\ra = 0$ for all $Q$ large enough, thus $(x, y) \notin E$. Suppose $x \notin \QQ$, and let $p_k/q_k$ be the convergents of $x$. Then, $|q_k x - p_k| \leq \frac{1}{q_{k+1}}$. Now suppose $a = a'/d$ and $b = b'/d$ where $a', b' \in \ZZ$ and $d \in \NN$; define also $q_k': = dq_k$. One has $\la q_k' x \ra \le \frac{d}{q_{k+1}} = \frac{d^2}{q_{k+1}'}$ and
        $$
        \la q_k' y \ra \leq |dq_k(ax + b) - (a'p_k + b'q_k)| = |a'(q_k x - p_k)| \le \frac{|a'|d}{q_{k+1}'}.
        $$
        Therefore for $q_k' \leq Q < q_{k+1}'$ one has 
        $$
        Q\cdot \min_{1\leq q\leq Q}\la qx\ra\la qy\ra \leq Q\cdot \la q_k'x\ra\la q_k'y\ra \leq \frac{|a'|d^3}{q_{k+1}'} \to 0
        $$
        as $k \to \infty$, and $(x,y) \notin E$.
    \end{enumerate}
\end{proof}

  \appendix \section{Independence of HAW on $\RR^2$ and HAW on curves}\label{appendix1}

    The ambient and curvewise HAW properties do not imply one another, even if
curvewise HAW is required for every curve satisfying the assumptions of Theorem \ref{T:main-curves}.

\begin{example}[HAW on curves does not imply HAW on $\RR^2$]
\label{Ex:curvewise-not-ambient-gasket}
Let
\[
 G=\left\{\sum_{m=1}^{\infty}3^{-m}d_m:
 d_m\in\{(0,0),(1,0),(0,1)\}\right\},
 \qquad K=\RR^2\setminus G.
\]
Then $G$ is the one-dimensional Sierpiński gasket studied by Kenyon \cite{Ken97}. Equivalently, it is
the attractor of the homotheties
\[
 f_d(z)=\frac{z+d}{3},
 \qquad d\in\{(0,0),(1,0),(0,1)\}.
\]
In particular, the defining iterated function system satisfies the strong separation condition.  
We will prove that:
\begin{enumerate}
    \item[(SG1)] the set $K$ is not HAW on $\RR^2$;
    \item[(SG2)] the set $K$ is HAW on every regular $C^2$ curve 
     such that
    \begin{equation}\label{E:gasket-condition}
        \text{the set $\mathcal{C}_{nz} : = \{ z \in \mathcal{C}: \kappa(z) \neq 0 \}$ is absolutely winning on $\mathcal{C}$.}
    \end{equation}
\end{enumerate}
\end{example}

\begin{proof}[Proof of (SG1)] Let $D=\operatorname{diam}G$, and put
\[
 \delta=\min_{|e|=1}\left(
       \max_{z\in G}\langle e,z\rangle
       -\min_{z\in G}\langle e,z\rangle\right).
\]
The minimum exists and is positive, because $G$ contains the three noncollinear
points $(0,0),(1/2,0),(0,1/2)$.  If
$L=\{z\in\RR^2:\langle e,z\rangle=c\}$ is an affine line, then by continuity, the range of
$\langle e,\cdot\rangle$ on $G$ has length at least $\delta$. So there exists some $z\in G$ satisfying $\operatorname{dist}(z,L)\ge \delta/2$.

Given $z\in G$ and $0<\rho<D$, we choose a cylinder
$(f_{d_1}\circ \cdots\circ f_{d_n})(G)$ 
containing the point $z$ such that
\[
 3^{-n}D\le\rho<3^{-n+1}D.
\]
So the cylinder is contained in $B(z,\rho)$.  Applying the preceding width estimate after the homothety $f_{d_1}\circ \cdots\circ f_{d_n}$, we see that for every affine line $L$, the cylinder contains a point whose distance from $L$ is at least
\[
 \frac{3^{-n}\delta}{2}>\frac{\delta}{6D}\rho.
\]
Thus $G$ is hyperplane diffuse in the sense of \cite[Definition~4.2]{BFKRW12}.  If $K$ were HAW in $\RR^2$, then \cite[Proposition~4.9]{BFKRW12} would make it HAW on $G$.  By \cite[Proposition~4.7 and Theorem~4.6]{BFKRW12}, $K\cap G$ would then have positive Hausdorff dimension, which is impossible since $K\cap G=\varnothing$. This completes the proof.
\end{proof}

\begin{proof}[Proof of \textup{(SG2)}]
Let $\mathcal C$ be a regular $C^2$ curve satisfying
\eqref{E:gasket-condition}.  We first show that \textup{(SG2)} follows
from
\begin{equation}\label{E:gasket-curve-slice}
 \dimH(\mathcal C_{nz}\cap G)=0.
\end{equation}
Indeed, let $\psi:V\to\mathcal C_{nz}$ be an arbitrary coordinate
parametrization.  Covering $V$ by countably many compact intervals on
which $\psi$ is bi-Lipschitz, \eqref{E:gasket-curve-slice} gives
\[
 \dimH\psi^{-1}(G)=0.
\]
Lemma~\ref{L:HAW-mnfd}\textup{(iv)} therefore shows that
$\psi^{-1}(K)$ is HAW on $V$.  Hence $K\cap\mathcal C_{nz}$ is HAW on
$\mathcal C_{nz}$.

Since curvature is continuous, $\mathcal C\setminus\mathcal C_{nz}$
is closed.  Moreover, $\mathcal C_{nz}$ is absolutely winning, and
hence HAW, on $\mathcal C$ by \eqref{E:gasket-condition}.
Lemma~\ref{L:HAW-open}
now shows that $K$ is HAW on $\mathcal C$.

It remains to prove \eqref{E:gasket-curve-slice}.  We first note that
\begin{equation}\label{E:gasket-line}
\text{for every affine line $L$ of irrational slope,} \ 
 \dimH(L\cap G)=0.
\end{equation}
Indeed, write the slope of $L$ as $u\in\RR\setminus\QQ$.  The map
\[
 \Pi_{-u^{-1}}:\RR^2\to\RR,\qquad
 \Pi_{-u^{-1}}(x,y)=x-u^{-1}y,
\]
is constant on $L$.  Since $L\cap G$ is Borel and
$-u^{-1}\notin\QQ$, \cite[Corollary~6.3]{Shm19} gives
\[
 \dimH(L\cap G)
 =\dimH\Pi_{-u^{-1}}(L\cap G)=0.
\]

Now cover $\mathcal C_{nz}$ by countably many compact embedded
coordinate arcs, and let
$\phi:J\to\mathcal C_{nz}$ be a regular parametrization of one such
arc.  For $v\in\ZZ^2\setminus\{0\}$ put
\[
 g_v(t)=\det(\phi'(t),v).
\]
If $g_v(t_0)=0$, then $\phi'(t_0)=a v$ for some $a\ne0$.  Since the
curvature does not vanish on $\mathcal C_{nz}$,
\[
 g_v'(t_0)
 =\det(\phi''(t_0),v)
 =-\frac{\det(\phi'(t_0),\phi''(t_0))}{a}
 \ne0.
\]
Thus every zero of $g_v$ is isolated, and consequently its zero set is
countable.  Taking the union over
$v\in\ZZ^2\setminus\{0\}$ shows that the tangent direction of this arc
is rational at only countably many points.  Lemma~\ref{L:Shmerkin} (see below) therefore gives
$ \dimH(\phi(J)\cap G)=0$.
The countable arc cover proves \eqref{E:gasket-curve-slice}, completing
the proof.
\end{proof}

{\begin{lemma}\label{L:Shmerkin}
    Let $G\subseteq \RR^2$ be the one-dimensional Sierpiński gasket, and let $\mathcal{C}$ be a regular $C^1$ curve whose tangent directions are rational only at points in a zero-dimensional set. Then $\dimH(\mathcal{C}\cap G)=0$.
\end{lemma}
 \begin{proof}
Suppose that $\dimH(\mathcal{C}\cap G)>0$. We choose a regular parametrization $\phi:J\to\mathcal C$ of one arc, 
and write $Z=\{t\in J:\Span_{\RR}\{\phi'(t)\}\in \PP^1(\QQ)\}$. Since the set $Z$ has zero dimension, by the inner regularity of Frostman measure, there exists a compact subset $$
F\subset(\phi(J)\cap G)\setminus \phi(Z)
 \quad\text{with}\quad \dimH{F}>0.
$$

By \cite[Proposition~5.7]{KOR18}, one may choose a set $M$ with the following properties: \begin{itemize}
    \item[(M1)] $\dimH{M}\geq \dimH{F}$;
    \item[(M2)] $M$ is the limit of a sequence of sets $(M_n)_{n\geq 0}$ in the Hausdorff metric, where $$
    M_n\subset(\lambda_nF+z_n)\cap[0,1]^2,
 \quad \lambda_n\to +\infty, \quad \text{and}\quad z_n\in\RR^2.
    $$
\end{itemize}
Here the set $M$ is called a \textsl{Furstenberg microset} of $F$, and the sets $M_n$ are called \textsl{minisets} of $F$. 

Next we shall show that $M$ is contained in an affine line with irrational slope. Indeed, fix $w\in M$, choose
$w_n=\lambda_nf_n+z_n\in M_n$ with $w_n\to w$, and pass to a subsequence such that $f_n\to f\in F$.  For any $w'\in M$, choose
$w'_n=\lambda_nf'_n+z_n\in M_n$ with $w'_n\to w'$. So we have $$
\lambda_n(f_n-f_n')\to w-w'\quad \text{as}\quad n\to +\infty.
$$
Note that
$|f_n-f'_n|=O(\lambda_n^{-1})$. We write $f_n=\phi(t_n)$ and
$f'_n=\phi(t'_n)$. Since $\phi$ is bi-Lipschitz on the compact embedded arc, one has
\[
 t_n\to t:=\phi^{-1}(f),\quad \text{and}
 \quad |t_n-t_n'|=O(\lambda_n^{-1}).
\]
Taylor's expansion of $\phi$ on $J$ now gives
\[
 \lambda_n\left((f_n-f'_n)
 -(t_n-t'_n)\phi'(t_n)\right)\to 0\quad \text{as}\quad n\to +\infty.
\]
Since the scalars $\lambda_n(t_n-t_n')$ are bounded, by passing to a subsequence, we see that $$w-w'\in\RR\phi'(t).$$  
Thus, we conclude that $M$ is contained in an affine
line $L$ parallel to the tangent of the arc at $t$.  Since
$t\notin Z$, this line has irrational slope.

Now we consider the minisets of $G$:
\[N_n=(\lambda_nG+z_n)\cap[0,1]^2.\]
Since $M_n\subset N_n$, the sets $N_n$ are nonempty.  By compactness of the hyperspace of nonempty compact subsets of $[0,1]^2$, a further subsequence converges in the Hausdorff metric to a microset $N$ of $G$; moreover, $M\subset N$. Note that the gasket $G$ is self-homothetic and satisfies the strong separation condition. It follows from \cite[Proposition~5.4]{KR16} that any microset of $G$ can be realized as a miniset. Consequently, one has
\[
 M\subset N\subset(\lambda G+z)\cap[0,1]^2
\]
for some $\lambda\ge1$ and $z\in\RR^2$, so $G$ contains the set $\lambda^{-1}(M-z)$. 

Therefore, we conclude that the subset $\lambda^{-1}(M-z)\subseteq G$ has positive Hausdorff dimension and lies on a line of irrational slope. This contradicts \eqref{E:gasket-line}.
 \end{proof}
 }

\begin{remark}
    A general version of Lemma \ref{L:Shmerkin} was announced in \cite[Remark 8.3(c)]{Shm19} without proofs. For the sake of completeness, we provide the above proof.
\end{remark}

\begin{example}[HAW on $\RR^2$ does not imply HAW on all $C^2$ curves with nonzero curvature]
\label{Ex:ambient-not-curvewise}
Let
\[
 \Bad=\left\{a\in\RR:\inf_{q\in\NN}q\la qa\ra>0\right\},
 \qquad E=\RR\times\Bad = \Bad(0,1).
\]
As mentioned in Section \ref{SS:LCULC}, the set $E$ is HAW.

Let $f\in C^2(\RR)$ satisfy $f''(x)\neq 0$ for all $x\in\RR$. Consider $$
F(x,y)=(x,y-f(x)),\quad \text{and}\quad S=F^{-1}(E). 
$$
Since $\det DF(x,y)=1$, \cite[Theorem~2.4]{BFKRW12} shows that the set $S$ is HAW on $\RR^2$.
On the other hand, for any $a\in\RR$, the set $$
\mathcal{C}_a:=F^{-1}(\RR\times \{a\})=\{(x,f(x)+a):x\in\RR\}
$$
is a $C^2$ curve with curvature \[
 \frac{|f''(x)|}{(1+f'(x)^2)^{3/2}} \neq 0.
\]
When $a\notin\Bad$, one has $\mathcal{C}_a\cap S=\varnothing$, which is clearly not HAW on $\mathcal{C}_a$.
\end{example}

\section*{Acknowledgments} The authors are grateful to Lifan Guan, Dmitry Kleinbock, and Nikolay Moshchevitin for helpful suggestions. During the preparation of this work, the authors used GPT-5.6 Pro to assist in doing calculations and checking proofs. All mathematical arguments were subsequently verified independently by the authors, who take full responsibility for the correctness of the results.

\end{document}